\documentclass[11pt,reqno]{amsart}

\usepackage[english]{babel}
\usepackage{amsmath, amsthm, amssymb}
\usepackage{amsfonts}
\usepackage{mathrsfs}
\usepackage{amsrefs}
\usepackage{hyperref}
\usepackage{xcolor}
\usepackage{enumitem}
\usepackage{dsfont}

\theoremstyle{plain}
\newtheorem{theorem}{Theorem}[section]
\newtheorem{corollary}[theorem]{Corollary}
\newtheorem{lemma}[theorem]{Lemma}
\newtheorem{proposition}[theorem]{Proposition}

\theoremstyle{definition}
\newtheorem{definition}[theorem]{Definition}
\newtheorem{remark}[theorem]{Remark}

\numberwithin{equation}{section}

\newcommand{\hcal}{\mathcal{H}}
\newcommand{\R}{\mathbb{R}}
\newcommand{\N}{\mathbb{N}}
\newcommand{\Z}{\mathbb{Z}}
\newcommand{\jacobi}{J_u}
\newcommand{\vv}{b}
\newcommand{\bb}{\widehat{b}}
\newcommand{\id}{I}
\renewcommand{\div}{{\rm div}}
\newcommand{\tr}{{\rm tr}}
\renewcommand{\d}{\mathop{}\!\mathrm{d}}
\newcommand{\supp}{\, {\rm supp} \,}
\newcommand{\elliptic}{c_0}
\newcommand{\bounded}{C_0}
\newcommand{\cc}{\textrm{\bf c}}
\newcommand{\diam}{{\rm diam}}
\newcommand{\exit}{\vartheta}

\newcounter{desccount}

\newcommand{\descref}[1]{\hyperref[#1]{#1}}

\begin{document}
	
\title[Boundary H\"{o}lder continuity of stable solutions in $C^{1,1}$ domains]{Boundary H\"{o}lder continuity of stable solutions to semilinear elliptic problems in $C^{1,1}$ domains}
\author[I. U. Erneta]{I\~{n}igo U. Erneta}
    \address{I\~{n}igo U. Erneta. Department of Mathematics\\
    Rutgers University\\
110 Frelinghuysen Rd., Piscataway, NJ 08854, USA}
    \email{inigo.erneta@rutgers.edu}

\thanks{The author acknowledges financial support from MINECO grant MDM-2014-0445-18-1 through the Mar\'{i}a de Maeztu Program for Units of Excellence in R\&D.
He is additionally supported by Spanish grants MTM2017-84214-C2-1-P and PID2021-123903NB-I00 funded by MCIN/AEI/10.13039/501100011033 and by ERDF ``A way of making Europe''.
The author is also supported by Catalan project 2021 SGR 00087.
}

\begin{abstract}
This article establishes the boundary H\"{o}lder continuity of stable solutions to semilinear elliptic problems in the optimal range of dimensions $n \leq 9$, for $C^{1,1}$ domains. We consider equations $- L u = f(u)$ in a bounded $C^{1,1}$ domain $\Omega \subset \R^n$, with $u = 0$ on $\partial \Omega$, where $L$ is a linear elliptic operator with variable coefficients and $f \in C^1$ is nonnegative, nondecreasing, and convex. The stability of $u$ amounts to the nonnegativity of the principal eigenvalue of the linearized equation $- L - f'(u)$. Our result is new even for the Laplacian, for which [Cabr\'{e}, Figalli, Ros-Oton, and Serra, Acta Math. 224 (2020)] proved the H\"{o}lder continuity in $C^3$ domains.
\end{abstract}


\maketitle

\section{Introduction}

Let $\Omega \subset \R^n$ be a bounded domain and $f \colon \R \to \R$ a $C^1$ function.
In this paper, we consider stable solutions
$u\colon \overline{\Omega} \to \R$
to the semilinear problem
\begin{equation}\label{pb:omega}
\left\{
\begin{array}{cl}
- L u = f(u) & \text{ in } \Omega\\
u = 0 & \text{ on } \partial \Omega,
\end{array}
\right.
\end{equation}
where $L$ is a
second order linear elliptic differential operator
of the form
\begin{equation}
\label{def:op}
L u = a_{ij}(x) u_{ij} + \vv_i(x) u_i, \quad a_{ij}(x) = a_{ji}(x).
\end{equation}
We assume that the coefficient matrix $A(x) = (a_{ij}(x))$ is uniformly elliptic in $\Omega$, i.e., there are positive constants $\elliptic$, $\bounded$ such that 
\begin{equation}
\label{elliptic}
\elliptic |p|^2 \leq a_{ij}(x) p_i p_j \leq \bounded |p|^2 \quad \text{ for all } p \in \R^n.
\end{equation}
This last condition is denoted by $\elliptic \leq A(x) \leq \bounded$.
In addition, we will always assume that 
\begin{equation}
\label{reg:coeffs}
a_{ij} \in C^{0,1}(\overline{\Omega}), \quad \vv_i \in L^{\infty}(\Omega) \cap C^0(\Omega).
\end{equation}
For some auxiliary results, we will further need that
\begin{equation}
\label{reg:b}
\vv_i \in C^{0}(\overline{\Omega}).
\end{equation}

A strong solution $u \in C^{0}(\overline{\Omega}) \cap W^{2,n}_{\rm loc}(\Omega)$ of \eqref{pb:omega} is \emph{stable} if the principal eigenvalue of the linearized equation at $u$ is nonnegative.
Equivalently, the solution $u$ is stable if there is a function $\varphi \in W^{2,n}_{\rm loc}(\Omega)$ satisfying
\begin{equation}
\label{stable:point}
\left\{
\begin{array}{cl}
\jacobi \varphi \leq 0 &\text{ a.e. in } \Omega,\\
\varphi > 0 &\text{ in } \Omega,
\end{array}
\right.
\end{equation}
where $\jacobi := L + f'(u)$ denotes the Jacobi operator (the linearization) at $u$.
For variational problems, the stability condition amounts to the nonnegativity of the second variation.
In particular, the class of stable solutions contains (local and global) minimizers.
Instead, here we are interested in non-variational equations as above.
For fundamental properties of the principal eigenvalue of linear non-selfadjoint operators such as $\jacobi$, we refer to the classic work of Berestycki, Nirenberg, and Varadhan~\cite{BerestyckiNirenbergVaradhan}.

The aim of this article is to investigate the regularity up to the boundary of stable solutions to~\eqref{pb:omega}.
Here, the question reduces to showing that solutions are bounded, since the linear theory then allows to prove further smoothness properties.
Our present work extends the boundary regularity results of Cabr\'{e}, Figalli, Ros-Oton, and Serra~\cite{CabreFigalliRosSerra} and of Cabr\'{e}~\cite{CabreRadial} for the Laplacian to the above operators with variable coefficients.
When $n \geq 10$, examples of singular stable solutions have been known for a long time.
In the breakthrough article~\cite{CabreFigalliRosSerra}, the authors solved the long-standing conjecture: if $n \leq 9$, then all stable solutions are bounded (when $L = \Delta$).
Their proof was quite delicate and relied on a contradiction-compactness argument which did not allow to quantify the constants in the estimates.
An alternative quantitative proof has been recently found in~\cite{CabreRadial}, although it only applies to the Laplacian in flat domains.
Generalizing and combining ideas form these two works, we will give, for the first time (even for the Laplacian), a quantitative proof valid in curved domains.

For $L = \Delta$, a key assumption needed in~\cite{CabreFigalliRosSerra} was for the domain to be of class $C^3$.
As we explain next, our analysis will allow us to weaken this condition to a $C^{1,1}$ regularity assumption.
For this, starting from a curved boundary, we flatten it out (locally) by a change of variables.
This transformation does not alter the semilinear nature of the equation but modifies the coefficients, which now involve first and second order derivatives of the flattening map.
Namely, the new coefficients $a_{ij}$ include first derivatives of this map, while the $\vv_i$ contain second derivatives.
Now, the crucial point is to obtain universal a priori estimates independent of the nonlinearity (in the spirit of~\cite{CabreFigalliRosSerra}) and having a specific dependence on the coefficients.
Our bounds will depend only on the ellipticity constants in \eqref{elliptic} and on the norms $\|\nabla a_{ij}\|_{L^{\infty}}$ and~$\|\vv_i\|_{L^{\infty}}$ of the coefficients, corresponding to the flattening map of a $C^{1,1}$ domain.
As mentioned above, thanks to a device of~\cite{CabreRadial} for flat domains, our estimates in the new coordinates will all be quantitative.
Combined with the interior bounds that we established in~\cite{ErnetaInterior}, they will yield a global estimate.

We believe that our ideas can also be applied to study the boundary regularity of stable solutions for other equations.
Our method provides a robust, direct way of proving quantitative estimates up to the boundary.
In particular, when $L$ is the $p$-Laplacian, we could extend the optimal interior bounds of Cabr\'{e}, Miraglio, and Sanch\'{o}n in~\cite{CabreMiraglioSanchon} up to the boundary.
By contrast, the previous work~\cite{CabreFigalliRosSerra} relies on an intricate blow-up and Liouville theorem argument.
The authors of~\cite{CabreFigalliRosSerra} need this in order to apply a result of theirs only available on a flat boundary, which they could only prove by contradiction-compactness.
This critical step does not allow them to quantify the constants in their inequalities.

Variational problems have a long history of regularity results for stable solutions, starting with the pioneering work of Crandall and Rabinowitz~\cite{CrandallRabinowitz} in the seventies.
For exponential and power nonlinearities, they showed that stable solutions are bounded in smooth domains when $n \leq 9$ 
(see also Joseph and Lundgren~\cite{JosephLundgren} for an exhaustive analysis of the radial case).
Their result is optimal, since the logarithm $u(x) = \log(1/|x|^2) \in W^{1,2}_0(\Omega)$ solves \eqref{pb:omega} (in the weak sense) with $\Omega = B_1$, $L = \Delta$, and $f(u) = 2(n-2)e^u$, and is stable for $n \geq 10$.
This last fact follows immediately from Hardy's inequality.
Surprisingly,~\cite{CrandallRabinowitz} appears to be the only variational paper where variable coefficients have been considered.
Namely, the a priori estimates in~\cite{CrandallRabinowitz} apply to self-adjoint operators in divergence form, with merely bounded coefficients.
However, the methods used cannot be extended to treat more general nonlinearities.

The motivation for considering exponential nonlinearities in~\cite{CrandallRabinowitz} came from problems in combustion theory, 
namely, from the so called explosion or Gelfand problem~\cite{Gelfand} (recalled in Subsection~\ref{subsection:gelfand} below).
In the nineties, Brezis~\cite{Brezis-IFT} asked whether the optimal dimension could be the same for more general nonlinearities.
He was interested in a natural class of nonlinearities for which the Gelfand problem admits stable solutions, namely:
nonnegative, nondecreasing, convex, and superlinear ones.
This question motivated a series of works trying to establish global a priori estimates for stable solutions 
to \eqref{pb:omega} in the model case, i.e., when $L = \Delta$.
First, Nedev~\cite{Nedev} proved their boundedness for $n \leq 3$.
Then, Cabr\'{e} and Capella~\cite{CabreCapella-Radial} reached the optimal dimension $n\leq 9$ in the radial case.
Later, Cabr\'{e}~\cite{Cabre-Dim4} and Villegas~\cite{Villegas} showed the boundedness when $n \leq 4$.
Afterwards, Cabr\'{e} and Ros-Oton~\cite{CabreRosOton-DoubleRev} proved the boundedness for $n \leq 7$ when $\Omega$ is a domain of double revolution.
Finally, Cabr\'{e}, Figalli, Ros-Oton, and Serra~\cite{CabreFigalliRosSerra} solved the conjecture in $C^3$ domains.

Concerning non-variational problems, there is only one paper, to the best of our knowledge, studying the regularity of stable solutions in our setting.
In~\cite{CowanGhoussoub}, Cowan and Ghoussoub consider operators of the form~\eqref{def:op}, assuming that $a_{ij} = \delta_{ij}$ and $\vv_i \in C^{\infty}(\overline{\Omega})$, where $\Omega \subset \R^n$ is a smooth bounded domain.
They showed that stable solutions are bounded when $n \leq 9$ for the exponential nonlinearity.
In particular, adding an advection term does not modify the optimal dimension in this case.
This is in accordance with our present work, where our mild smoothness assumption on the coefficients guarantees the invariance of the optimal dimension for general nonlinearities.
It is worth noting that the interest of the authors in~\cite{CowanGhoussoub} was in singular nonlinearities appearing in the modeling of MEMS devices, namely, nonlinearities $f = f(u)$ defined for $u \in [0,1)$ which blow up at $u = 1$.

Finally, we would also like to mention the recent work of Costa, de Souza, and Montenegro~\cite{CostadeSouzaMontenegro} for a more general non-variational setting.
In that paper, the authors consider the Gelfand problem for systems of equations including operators of the form~\eqref{def:op}.
While they are mostly concerned with the existence of stable solutions to this problem, they also address the question of regularity, but only for the Laplacian.
More precisely, they study stable solutions $u \colon \R^n \to \R^m$ (with $m \geq 2$) of $- \Delta u = F(u)$, where $F \colon \R^m \to \R^m$ satisfies natural assumptions analogous to ours.
In the radial case, they are able to show that they are bounded for $n \leq 9$, which is the optimal dimension for scalar equations.
For convex $C^{1,1}$ domains, they show their boundedness for $n \leq 3$ by adapting the interior estimates of Cabr\'{e}~\cite{Cabre-Dim4}.
We believe that our techniques can also be used to reach the optimal dimension for systems of equations in any $C^{1,1}$ domain.

\subsection{Main results}

Our main result provides two types of a priori estimates for strong stable solutions on domains of class $C^{1,1}$.
The first one is an energy estimate valid in all dimensions.
It was announced in our previous paper~\cite{ErnetaBdy1}, 
where we proved it in flat domains (see Theorem~\ref{thm:previous} below).
Here, we will complete the proof, which involves a covering and approximation procedure.
The second estimate is a bound of the H\"{o}lder norm in the optimal range of dimensions $n \leq 9$.
Here and throughout the paper,
when we write $C = C(\ldots)$ for a constant $C$ we mean that $C$ depends only on the quantities appearing inside the parentheses.

\begin{theorem}
\label{thm:c11}
Let $\Omega \subset \R^n$ be a bounded domain of class $C^{1,1}$,
let $L$ satisfy conditions~\eqref{elliptic} and~\eqref{reg:coeffs} in $\Omega$,
and let $f \in C^1(\R)$ be nonnegative, nondecreasing, and convex.
Let $u \in C^{0}(\overline{\Omega}) \cap W^{2,n}_{\rm loc}(\Omega)$ be a nonnegative stable solution of $- L u = f(u)$ in $\Omega$, with $u = 0$ on~$\partial \Omega$.

%

Then
\begin{equation}
\label{c11:hint}
\|\nabla u\|_{L^{2+\gamma}(\Omega)} \leq C \|u\|_{L^1(\Omega)},
\end{equation}
where $\gamma = \gamma(n) > 0$ and 
$C = C(\Omega, n, \elliptic, \bounded, \| \nabla a_{ij}\|_{L^{\infty}(\Omega)}, \|b_i\|_{L^{\infty}(\Omega)})$.
In addition,
\begin{equation}
\label{c11:holder}
\|u\|_{C^{\alpha}(\overline{\Omega})} \leq C \|u\|_{L^1(\Omega)} \quad \text{ if } n \leq 9,
\end{equation}
where $\alpha = \alpha(n, \elliptic, \bounded) > 0$ and 
$C = C(\Omega, n, \elliptic, \bounded, \| \nabla a_{ij}\|_{L^{\infty}(\Omega)}, \|b_i\|_{L^{\infty}(\Omega)})$.
\end{theorem}

The proof of Theorem~\ref{thm:c11} relies 
on analogous boundary estimates in half-balls, given next,
as well as in the interior bounds that we obtained in~\cite{ErnetaInterior}.
Below, we always write $\R^n_+ = \{x \in \R^n \colon x_n > 0\}$ and, for each $\rho > 0$, we let
\[
B_{\rho}^{+} := \R^n_+ \cap B_{\rho}.
 \]
Moreover, for any open 
set $\Omega \subset \R^n_+$, we denote its lower and upper boundaries by
\[
\partial^{0} \Omega = \{x_n = 0\} \cap \partial \Omega, \quad \partial^{+}\Omega = \R^n_+ \cap \partial \Omega.
\]

The H\"{o}lder estimate~\eqref{c11:holder} will be a consequence of the following:
\begin{theorem}
\label{thm:holder}
Let $L$ satisfy conditions~\eqref{elliptic},~\eqref{reg:coeffs}, and~\eqref{reg:b} in $\Omega = B_1^{+}$,
and let $f \in C^1(\R)$ be nonnegative, nondecreasing, and convex.
Let $u \in W^{3,p}(B_1^{+})$, for some $p > n$, be a nonnegative stable solution to $- L u = f(u)$ in $B_1^{+}$, with $u = 0$ on $\partial^0 B_1^{+}$.
%

Then 
\begin{equation}
\label{holder}
\|u\|_{C^{\alpha}(\overline{B_{1/2}^{+}})} \leq C \|u\|_{L^1(B_{1}^{+})} \quad \text{ if } n \leq 9,
\end{equation}
where $\alpha = \alpha(n, \elliptic, \bounded) > 0$ and $C = C(n, \elliptic, \bounded, \|\nabla a_{ij} \|_{L^{\infty}(B_{1}^{+})}, \|\vv_i \|_{L^{\infty}(B^{+}_1)})$.
\end{theorem}

\begin{remark}
In contrast to Theorem~\ref{thm:c11} above, here we require additional hypotheses on the solution and the coefficients.
Namely, we need third weak derivatives of $u$ to remove the nonlinearity from the stability condition, making our bounds independent of $f$.
We further need the continuity of $\vv_i$ up to the boundary, assumption \eqref{reg:b}, to control certain surface integrals over $\partial^0 B_1^+$ appearing in the proof.

To prove the estimate in $C^{1,1}$ domains from the one in half-balls, we carry out an approximation and flattening procedure.
We consider an exhaustion by smooth domains $\Omega_k$ and, in each of them, we construct a stable solution $u_k$ to a semilinear equation with more regular coefficients.
The smoothness of the data will guarantee that $u_k$ is in $W^{3,p}(\Omega_k)$ and hence, flattening the boundary $\partial \Omega_k$, we may apply Theorem~\ref{thm:holder}.
Thanks to the $C^{1,1}$ regularity assumption, the constants in the bounds for $u_k$ in half-balls will be independent of $k$.
By convexity of $f$, the functions $u_k$ converge to the original solution and taking limits we deduce the theorem.
\end{remark}

As mentioned above, the energy estimate \eqref{c11:hint} in $C^{1,1}$ domains uses the analogous result in half-balls. 
In the following result from~\cite{ErnetaBdy1}, we obtained such a bound via Hessian estimates for stable solutions in the spirit of Sternberg and Zumbrun~\cite{SternbergZumbrun1}.
To prove Theorem~\ref{thm:holder}, we will need both the energy and Hessian estimates.
\begin{theorem}[\cite{ErnetaBdy1}]
\label{thm:previous}
Let $L$ satisfy conditions~\eqref{elliptic},~\eqref{reg:coeffs}, and~\eqref{reg:b} in $\Omega = B_1^{+}$,
and let $f \in C^1(\R)$ be nonnegative and nondecreasing.
Let $u \in W^{3,p}(B_1^{+})$, for some $p > n$, be a nonnegative stable solution to $- L u = f(u)$ in $B_1^{+}$, with $u = 0$ on $\partial^0 B_1^{+}$.

Then 
\begin{equation}
\label{hessianw}
\||\nabla u| \, D^2 u\|_{L^{1}(B_{1/2}^{+})} \leq C \|\nabla u\|_{L^2(B_1^{+})}^2,
\end{equation}
and
\begin{equation}
\label{higherint}
\|\nabla u\|_{L^{2+ \gamma}(B_{1/2}^{+})} \leq C \|u\|_{L^1(B_1^{+})},
\end{equation}
where $\gamma = \gamma(n) > 0$ and 
$C = C(n, \elliptic, \bounded, \|\nabla a_{ij}\|_{L^{\infty}(B_{1}^{+})}, \|\vv_i \|_{L^{\infty}(B^{+}_1)})$.
\end{theorem}

\subsection{Application: Regularity of the extremal solution in $C^{1,1}$ domains}
\label{subsection:gelfand}

Let $f \colon [0, +\infty) \to \R$ satisfy $f(0) > 0$ and be nondecreasing, convex, and superlinear at $+\infty$,
meaning that
\[
\lim_{u \to +\infty} \frac{f(u)}{u} = + \infty.
\]
Given a constant $\lambda > 0$, we consider the problem
\begin{equation}\label{pb:gelfand}
\left\{
\begin{array}{cl}
- L u = \lambda f(u) & \text{ in } \Omega\\
u > 0 & \text{ in } \Omega\\
u = 0 & \text{ on } \partial \Omega,
\end{array}
\right.
\end{equation}
where $\Omega \subset \R^n$ is a $C^{1,1}$ bounded domain, and $L$ satisfies conditions~\eqref{elliptic} and~\eqref{reg:coeffs} in $\Omega$.

The boundary value problem~\eqref{pb:gelfand} is the \emph{Gelfand problem} mentioned above.
It was first presented by Barenblatt in a volume edited by Gelfand~\cite{Gelfand}.
Originally, \eqref{pb:gelfand} was introduced to study ignition and explosion phenomena in the theory of thermal combustion.
In that framework, $u$ can be understood as the temperature of a combustible mixture, 
while $\lambda$ measures the relative strength of the reaction $f(u)$ with respect to the diffusion-advection processes modeled by $L$.
When $\lambda$ is large, solutions are not expected to exist, which is interpreted as the occurrence of an explosion.

Stable solutions play a prominent role in the Gelfand problem, as evidenced by the next proposition below.
For an account of the history and references for \eqref{pb:gelfand}, we refer the reader to the monograph of Dupaigne~\cite{Dupaigne}.
Here, instead, we only recall a basic, well-known result concerning the existence of solutions to~\eqref{pb:gelfand}.

Note that by the identification $a_{ij} \in C^{0,1}(\overline{\Omega}) = W^{1, \infty}(\Omega)$ we can always write our operator $L$ in divergence form 
\begin{equation}
\label{op:nondiv}
L u = \div \left( A(x) \nabla u\right) + \bb(x) \cdot \nabla u,
\end{equation}
where $\bb(x) = (\bb_{i}(x))$ is the vector field given by 
\begin{equation}
\label{def:bvector}
\bb_i(x) = \vv_i(x) - \partial_k a_{ki}(x). 
\end{equation}
In particular, since we always assume that $\partial_k a_{ij} \in L^{\infty}(\Omega)$ and $\vv_i \in L^{\infty}(\Omega)$, we also have $\bb_i \in L^{\infty}(\Omega)$.
The following result for non-variational problems has appeared in a slightly different form in~\cite{BerestyckiKiselevNovikovRyzhik,CostadeSouzaMontenegro}.
For the classical variational version, see, for instance,~\cite{Brezis-IFT, Dupaigne}.

\begin{proposition}
There exists a constant $\lambda^{\star} \in (0, +\infty)$ such that:
\begin{itemize}
\item[\rm{(i)}] For each $\lambda \in (0, \lambda^{\star})$ there is a unique strong stable solution $u_{\lambda} \in C^0(\overline{\Omega}) \cap W^{2,n}_{\rm loc}(\Omega)$ of \eqref{pb:gelfand}.
Moreover, we have $u_{\lambda} < u_{\lambda'}$ in $\Omega$ for $\lambda < \lambda'$.
\item[\rm{(ii)}] For $\lambda > \lambda^{\star}$ there is no strong solution.
\end{itemize}
Assume moreover that
\begin{equation}
\label{muchreg}
\div \, \bb \in L^{\infty}(\Omega),
\end{equation}
so that the adjoint operator
\[
L^{T} \zeta = 
\div\left( A(x) \nabla \zeta\right) - \bb(x) \cdot \nabla \zeta - \div \, \bb(x) \, \zeta
\]
is well defined for $\zeta \in W^{2,n}_{\rm loc}(\Omega)$ and has bounded coefficients.
Then:
\begin{enumerate}[label={\rm (iii)}]
\item \label{more:cond}
For $\lambda = \lambda^{\star}$ there exists a unique $L^1$-weak solution $u^{\star}$, in the following sense:
$u^{\star} \in L^{1}(\Omega)$, $f(u^{\star}){\rm dist}(\cdot, \partial \Omega) \in L^1(\Omega)$, and
\[
- \int_{\Omega} u^{\star} L^{T} \zeta \d x = \lambda^{\star} \int_{\Omega} f(u^{\star}) \zeta \d x 
\]
for all $\zeta \in W^{2,n}(\Omega)$ with $L^T \zeta \in L^{\infty}(\Omega)$ and $\zeta|_{\partial \Omega} = 0$.

The solution $u^{\star}$ is called the \emph{extremal solution} of \eqref{pb:gelfand} and satisfies $u_{\lambda} \uparrow u^{\star}$ as $\lambda \uparrow \lambda^{\star}$.
\end{enumerate}
\end{proposition}
\begin{remark}
The uniqueness of $u^{\star}$ is due to Martel~\cite{Martel}.
Although he proved it in the model case $L = \Delta$,
the same ideas extend to the operators considered in this paper.
\end{remark}

\begin{remark}
The additional regularity~\eqref{muchreg} of the drift $\bb_i$ is needed in \ref{more:cond} to guarantee that 
$u^{\star} \in L^1(\Omega)$ and $f(u^{\star}){\rm dist}(\cdot, \partial \Omega) \in L^1(\Omega)$.
For this, testing the equation with the principal eigenfunction $\phi$ of $L^T$, by superlinearity of $f$, it is easy to show that
\[
\int_{\Omega} f(u_{\lambda}) \phi \d x \leq C \quad \text{ for } \lambda \in (0, \lambda^{\star}),
\]
where $C$ does not depend on $\lambda$.
By regularity 
$\phi \in W^{2, p}(\Omega)$ for all $p < \infty$, hence 
$\phi \in C^1(\overline{\Omega})$,
and by maximum principle $\phi \geq c \, {\rm dist}(\cdot, \partial\Omega)$,
whence $f(u^{\star}){\rm dist}(\cdot, \partial \Omega) \in L^1(\Omega)$.
Now, testing the equation with the unique solution to $- L^T \exit = 1$ in $\Omega$, $\exit = 0$ on $\partial \Omega$, we~also~have
\[
\int_{\Omega} u_{\lambda} \d x = \lambda \int_{\Omega} f(u_{\lambda}) \exit \d x \quad \text{ for } \lambda \in (0, \lambda^{\star}).
\]
Again, by regularity $\exit \leq C \phi$, and using the inequality above, we conclude that $u^{\star} \in L^1(\Omega)$.
\end{remark}

Since, a priori, the extremal solution $u^{\star}$ is only in $L^1(\Omega)$,
it is natural to investigate further regularity properties.
In this direction, 
Brezis and V\'{a}zquez~\cite[Problem~1]{BrezisVazquez} 
asked whether the extremal solution for the model operator $L = \Delta$ was always in $W^{1,2}_{0}(\Omega)$,
this being the natural energy space of the variational problem.
Similarly, as explained above, Brezis~\cite[Open problem~1]{Brezis-IFT} 
asked if the extremal solution $u^{\star}$ was always bounded for $n \leq 9$.
Both of these questions were answered positively by Cabr\'{e}, Figalli, Ros-Oton, and Serra~\cite{CabreFigalliRosSerra}
for the Laplacian in $C^3$ domains.
For this, they apply their a priori estimates to the classical stable solutions $\{u_{\lambda}\}_{\lambda < \lambda^{\star}}$ and, using that they are bounded in $L^1(\Omega)$, they pass to the limit as $\lambda \uparrow \lambda^{\star}$.
By this same procedure, our main theorem, Theorem~\ref{thm:c11},
extends their result to operators with coefficients in $C^{1,1}$ domains:
\begin{corollary}
Let $\Omega \subset \R^n$ be a bounded domain of class $C^{1,1}$, let $L$ satisfy conditions~\eqref{elliptic}, \eqref{reg:coeffs}, and~\eqref{muchreg} in $\Omega$, and let $f \in C^1(\R)$ be positive, nondecreasing, convex, and superlinear at $+\infty$.


Then the extremal solution $u^{\star}$ to \eqref{pb:gelfand} is in $W^{1,2+\gamma}_{0}(\Omega)$
for some $\gamma = \gamma(n) > 0$.
Moreover, if $n \leq 9$, then $u^{\star}$ is bounded (and hence is a strong solution 
in $W^{2,p}(\Omega)$ for all $p < \infty$).
\end{corollary}

\subsection{Structure of the proof}

By approximation, it will suffice to prove Theorem~\ref{thm:c11} in smooth domains.
Flattening the boundary, we further reduce the problem to half-balls
and, hence, the core of the proof is to show the H\"{o}lder estimate in Theorem~\ref{thm:holder}.
Moreover,
by a scaling and covering argument, 
we may assume that the operator $L$ is close to the Laplacian, i.e., the coefficients satisfy $A(0) = \id$ and $\|D A \|_{L^{\infty}} + \|\vv\|_{L^{\infty}} \leq \varepsilon$, with $\varepsilon$ small.

As in our previous works~\cite{ErnetaInterior, ErnetaBdy1},
to obtain a priori estimates,
we will use the stability of the solution via a more convenient integral inequality.
Following the notation in those papers,
we denote the norm induced by the positive definite matrix $A(x) = (a_{ij}(x))$ by
\[ |p|_{A(x)} := \left( a_{ij}(x) p_i p_j\right)^{1/2} \quad \text{ for } p \in \R^n. \]
In~\cite{ErnetaInterior}, we showed that if $u$ is stable, 
then 
\begin{equation}
\label{ineq:stable}
\int_{\Omega} f'(u) \xi^2 \d x 
\leq \int_{\Omega} 
\left|\nabla \xi-{\textstyle\frac{1}{2}} \xi A^{-1}(x) \bb(x)\right|^2_{A(x)}
\d x 
\quad \text{ for all } \xi \in C^\infty_c(\Omega),
\end{equation}
where the vector field $\bb(x) = (\bb_i(x))$
has been introduced in \eqref{def:bvector} above.

Using the integral stability inequality~\eqref{ineq:stable}
and thanks to the energy and Hessian estimates in Theorem~\ref{thm:previous},
we will be able to prove two key auxiliary results:
Propositions~\ref{prop:rad} and~\ref{prop:l1rad} below, which we comment on next.
Combined, they will yield the H\"{o}lder estimate~\eqref{holder} in~Theorem~\ref{thm:holder}.

The first proposition provides a weighted $L^2$ estimate for the radial derivative
\[
u_r = \frac{x}{|x|} \cdot \nabla u \quad (r = |x|)
\]
in a half-ball by the $L^2$ norm of the full gradient in a half-annulus,
under a smallness condition on the coefficients.
This bound requires $n \leq 9$ and will be essential in the proof of the H\"{o}lder regularity of stable solutions.
Here and throughout this work, a constant depending only on $n$, $\elliptic$, and $\bounded$ will be called \emph{universal}.

\begin{proposition}
\label{prop:rad}
Let $u \in W^{3,p}(B_1^{+})$,
for some $p > n$,
be a nonnegative stable solution of 
$- L u = f(u)$ in $B_{1}^{+}$, with $u = 0$ on $\partial^0 B_{1}^{+}$.
Assume that $f \in C^1(\R)$ is nonnegative and nondecreasing.
Assume that $L$ satisfies conditions~\eqref{elliptic},~\eqref{reg:coeffs}, and~\eqref{reg:b} in $\Omega = B_1^{+}$.

If $3 \leq n \leq 9$ and
\[
\|D A\|_{L^{\infty}(B_{1}^{+})} + \|\vv\|_{L^{\infty}(B_{1}^{+})} \leq \varepsilon_0,
\]
then 
\[
\int_{B_{\rho}^{+}} r^{2-n} u_r^2 \d x \leq C \int_{B_{2\rho}^{+} \setminus B_{\rho}^{+}} r^{2-n} |\nabla u|^2 \d x
+ C \varepsilon \int_{B_{4 \rho}^{+}} r^{3-n} |\nabla u|^2 \d x
\]
for all $\rho \leq 1/4$, where $\varepsilon_0 > 0$ and $C$ are universal constants.
%
\end{proposition}

Although Proposition~\ref{prop:rad} requires $n \geq 3$, 
adding superfluous variables, we will be able to use it to prove the $C^{\alpha}$ estimate when $n \leq 2$ as well.

To prove Proposition~\ref{prop:rad}, we use the integral stability inequality~\eqref{ineq:stable} with appropriate test functions.
Letting $\xi = \cc \eta$ in~\eqref{ineq:stable} with $\Omega = B_1^+$,
where $\cc$, $\eta$ are smooth functions satisfying $\cc = 0$ on $\partial^{0} B_{1}^{+}$ and $\supp \eta \subset B_1$,
if we integrate by parts, then \eqref{ineq:stable} becomes
\begin{equation}
\label{stab:half:jac}
\int_{B_{1}^{+}}  \cc \jacobi \cc \, \eta^2\d x 
\leq \int_{B_{1}^{+}} 
\cc^2 \left|\nabla \eta-{\textstyle\frac{1}{2}} \eta A^{-1}(x) \bb(x)\right|^2_{A(x)}
\d x.
\end{equation}
By approximation, we may choose
\[
\cc(x)
= x \cdot \nabla u(x) 
= r u_r \quad \text{ and } \quad \eta(x) = |x|_{A^{-1}(0)}^{\frac{2-n}{2}}\zeta(x),
\]
where $\zeta \in C^{\infty}_c(B_1)$ is a cut-off.
A test function of this type appeared for the first time in the work of Cabr\'{e} and Capella~\cite{CabreCapella-Radial} 
for the Laplacian in the radial case.
A similar choice was used in~\cite{CabreFigalliRosSerra} to establish the boundedness in $C^3$ domains.
Our function is a linear transformation of the latter one.

This choice of test function will lead to the desired inequality, but also produces weighted Hessian errors that we must control.
For this, we invoke Theorem~\ref{thm:previous}, which is why we need the continuity of the coefficient $\vv_i \in C^0(\overline{B_1^+})$ and the assumptions $f \geq 0$ and $f' \geq 0$ on the nonlinearity.
By contrast, for the Laplacian in $C^3$ domains, no such errors arise, which is why the previous works~\cite{CabreFigalliRosSerra, CabreRadial} did not need any assumptions on $f$ at this step.

In the second proposition, we control the $L^1$ norm of our stable solution by the $L^1$ norm of its radial derivative.
This estimate is an extension of a device in~\cite{CabreRadial} for the Laplacian in half-balls, and it is also the key step which makes our proofs quantitative.
In addition to a smallness condition on the coefficients, this is the only place where we need the assumption $A(0) = \id$ and where we use the convexity of $f$ crucially (aside from the approximation argument for $C^{1,1}$ domains).
It is worth mentioning that such a tool is also available for interior estimates, where the proof is entirely different and only requires $f$ to be nonnegative; see~\cite{CabreRadial, ErnetaInterior}.

\begin{proposition}
\label{prop:l1rad}
Let $u \in W^{3,p}(B_1^+)$, for some $p > n$,
be a nonnegative stable solution of $-L u = f(u)$ in $B_{1}^{+}$,
with $u = 0$ on $\partial^{0} B_{1}^{+}$.
Assume that $f \in C^1(\R)$ is nonnegative, nondecreasing, and convex.
Assume that $L$ satisfies conditions~\eqref{elliptic},~\eqref{reg:coeffs}, and~\eqref{reg:b} in $\Omega = B_1^{+}$.

If 
\[
A(0) = \id \quad \text{ and } \quad \|D A\|_{L^{\infty}(B_{1}^{+})} + \|\vv\|_{L^{\infty}(B_{1}^{+})} \leq \varepsilon_0,
\]
then
\[
\|u\|_{L^{1}(B_{1}^{+} \setminus B_{1/2}^{+})} \leq C \|u_r\|_{L^1(B_1^{+} \setminus B_{1/2}^{+})},
\]
where $\varepsilon_0 > 0$ and $C$ are universal constants.
\end{proposition}

The proof of the H\"{o}lder estimate~\eqref{holder} in Theorem~\ref{thm:holder} requires the previous results.
Combining the energy estimate~\eqref{higherint} with Proposition~\ref{prop:l1rad} on dyadic annuli and rescaling,
it is easy to show that the weighted Dirichlet energy in a ball is controlled by the weighted $L^2$ norm of the radial derivative in a larger ball.
Applying Proposition~\ref{prop:rad} and by hole-filling, we then deduce a decay of the former quantity, leading to a $C^{\alpha}$ bound.

\subsection{Outline of the article}

In Section~\ref{section:radial}, we prove Proposition~\ref{prop:rad},
the weighted inequality for the radial derivative.
Section~\ref{section:l1rad} focuses on the proof of Proposition~\ref{prop:l1rad}, controlling the solution by its radial derivative.
In Section~\ref{section:holder}, we obtain the H\"{o}lder estimates in Theorem~\ref{thm:holder}.
Finally, in Section~\ref{section:approximation} we prove our main result, Theorem~\ref{thm:c11}.

In Appendix~\ref{app:approximation}, we explain how to approximate $C^{1,1}$ domains from the interior by smooth sets satisfying uniform bounds.
Appendix~\ref{app:uniqueness} is devoted to proving the uniqueness of stable solutions for convex nonlinearities, an auxiliary result needed in our proof of the main theorem.

\section{The boundary weighted $L^2$ estimate for radial derivatives}
\label{section:radial}

Here we obtain the weighted estimates in half-balls for the radial derivative leading to Proposition~\ref{prop:rad}.
For this, we test the stability inequality \eqref{stab:half:jac} with the functions
\[
\cc := x \cdot \nabla u \quad \text{ and } \quad \eta = r^{\frac{2-n}{2}} \zeta,
\]
where $\zeta \in C^{\infty}_c(B_1)$ is a cut-off.
This is a valid test function, since 
$x \cdot \nabla u = 0$ on $\partial^0 B_1^+$.
Thus, the proof is essentially the same as in the interior case; see \cite{ErnetaInterior}.

First we test with a generic cut-off $\eta$, not necessarily a power function:
\begin{lemma}\label{lemma:bdy:radial}
Let $u \in W^{3,p}(B_1^+)$, for some $p > n$, be a stable solution of $- L u = f(u)$ in $B_{1}^{+}$, with $u = 0$ on $\partial^0 B_{1}^{+}$, for some function $f \in C^1(\R)$.
Assume that $L$ satisfies conditions~\eqref{elliptic} and~\eqref{reg:coeffs} in $\Omega = B_1^{+}$, and that
\[
\|D A\|_{L^{\infty}(B_{1}^{+})} + \|\vv\|_{L^{\infty}(B_{1}^{+})} \leq \varepsilon
\]
for some $\varepsilon > 0$.

Then
\[
\begin{split}
& \int_{B_{1}^{+}}  |\nabla u|_{A(x)}^2 \Big((n-2)\eta^2 +  x\cdot \nabla (\eta^2)\Big)\d x \\
& \qquad \qquad + \int_{B_{1}^{+}} \Big( -2 (x\cdot \nabla u) A(x)\nabla u \cdot \nabla (\eta^2)-|x \cdot \nabla u|^2 |\nabla \eta|_{A(x)}^2 \Big) \d x\\
& \qquad \qquad \qquad \qquad \leq
\displaystyle 
C \varepsilon \int_{B^{+}_1} |D^2 u| |\nabla u| |x|^2 \eta^2 \d x\\
 & \qquad \qquad \qquad \qquad \qquad \qquad +C \varepsilon \int_{B_1^+} |\nabla u|^2 \Big(|x|^{2} |\nabla (\eta^2)|  + (|x| + \varepsilon |x|^2) \eta^2 \Big) \d x,
\end{split}
\]
for all $\eta \in C_c^{\infty}(B_1)$,
where $C$ is universal.
\end{lemma}
\begin{proof}
We test the stability inequality \eqref{stab:half:jac} with $\cc = x \cdot \nabla u$.
First, we compute the left-hand side of \eqref{stab:half:jac}, i.e., 
$\int_{B_1^+} (x \cdot \nabla u) \jacobi [x \cdot \nabla u] \, \eta^2 \d x$.
Computing, we have
\begin{equation}
\label{reo1}
L[x \cdot \nabla u] = x_k a_{ij}(x) u_{ijk} + 2 a_{ij}(x) u_{ij} + x_k \vv_i(x) u_{ik} + \vv_i(x) u_{i} \quad \text{ for a.e. } x \in B_1^+.
\end{equation}
For the zero order term, integrating by parts and using the equation, we have
\begin{equation}
\label{reo2}
\begin{split}
&\int_{B_1^{+}} f'(u) (x \cdot \nabla u)^2 \, \eta^2 \d x\\  
& \qquad = \int_{B_1^{+}} \nabla [f(u)] \cdot (x \cdot \nabla u) x \, \eta^2 \d x
=  \int_{B_1^{+}} L u \,  \div\left\{ (x \cdot \nabla u) x \, \eta^2 \right\} \d x\\
& \qquad = - \int_{B_1^{+}} x \cdot \nabla [a_{ij}(x) u_{ij}] \, (x \cdot \nabla u) \, \eta^2  \d x + \int_{B_1^{+}} \vv_i(x) u_i \, \div\left\{ (x \cdot \nabla u) x \, \eta^2 \right\} \d x,
\end{split}
\end{equation}
where in the last line we integrated by parts again.
Combining \eqref{reo1} and \eqref{reo2}, it follows that
\begin{equation}
\label{reo3}
\begin{split}
&\int_{B_1^+} (x \cdot \nabla u) \jacobi[x \cdot \nabla u] \, \eta^2 \d x
= \int_{B_1^+} (x \cdot \nabla u) L [x \cdot \nabla u] \, \eta^2 \d x + \int_{B_1^+} f'(u)(x \cdot \nabla u)^2 \, \eta^2 \d x\\
& = \int_{B_1^+} 2 (x \cdot \nabla u) a_{ij}(x) u_{ij}  \eta^2 \d x - \int_{B_1^{+}} x \cdot \nabla a_{ij}(x) \, u_{ij} \, (x \cdot \nabla u) \, \eta^2  \d x \\
& \quad \quad +\int_{B_1^+} x_k \vv_i(x) u_{ik} (x \cdot \nabla u) \eta^2 \d x 
+ \int_{B_1^{+}}  \vv_i(x) u_{i} \Big(  (x \cdot \nabla u) \eta^2 + \div\left\{ (x \cdot \nabla u) x \, \eta^2 \right\} \Big)  \d x.
\end{split}
\end{equation}

Notice that the first integrand in the right-hand side of \eqref{reo3} can be written as
\begin{equation}
\label{reo4}
\begin{split}
2 (x \cdot \nabla u) a_{ij}(x) u_{ij} 
&= \div\left(2(x\cdot \nabla u) A(x)\nabla u - |\nabla u|^2_{A(x)} x\right) +(n-2) |\nabla u|^2_{A(x)}\\
& \quad \quad \quad - 2(x \cdot \nabla u) \partial_i a_{ij}(x) u_j + x \cdot \nabla a_{ij}(x) u_{i} u_{j}.
\end{split}
\end{equation}
Hence, substituting \eqref{reo4} in \eqref{reo3} we deduce
\begin{equation}
\label{reo5}
\begin{split}
&\int_{B_1^+} (x \cdot \nabla u) \jacobi[x \cdot \nabla u] \, \eta^2 \d x\\
&= \int_{B_1^+} \div\left(2(x\cdot \nabla u) A(x)\nabla u - |\nabla u|^2_{A(x)} x\right) \eta^2 \d x + \int_{B_1^+} (n-2) |\nabla u|^2_{A(x)} \eta^2 \d x
\\
& \quad \quad +\int_{B_1^+} \Big(-2 \partial_i a_{ij}(x) u_j \, (x \cdot \nabla u) + x \cdot \nabla a_{ij}(x) \, \big\{ u_{i} u_{j} -u_{ij} \, (x \cdot \nabla u) \big\} \Big) \eta^2 \d x\\
& \quad \quad +\int_{B_1^+} \vv_i(x) x_k u_{ik} (x \cdot \nabla u) \eta^2 \d x 
+ \int_{B_1^{+}}  \vv_i(x) u_{i} \Big(  (x \cdot \nabla u) \eta^2 + \div\left\{ (x \cdot \nabla u) x \, \eta^2 \right\} \Big)  \d x.
\end{split}
\end{equation}
Thus, integrating by parts the divergence term in \eqref{reo5} and applying the coefficient estimates $\|\nabla a_{ij}\|_{L^{\infty}} + \|\vv_i\|_{L^{\infty}} \leq \varepsilon$, we obtain the lower bound
\begin{equation}
\label{reo6}
\begin{split}
&\int_{B_1^+} (x \cdot \nabla u) \jacobi[x \cdot \nabla u] \, \eta^2 \d x\\
&\geq  - \int_{B_1^+} 2(x\cdot \nabla u) A(x)\nabla u \cdot \nabla (\eta^2) + \int_{B_1^+} |\nabla u|^2_{A(x)} \Big( (n-2)\eta^2+ x \cdot \nabla( \eta^2) \Big)  \d x\\
& \quad \quad 
-C \varepsilon \int_{B_1^+} \Big( |D^2 u | |\nabla u| |x|^2 \eta^2 + |\nabla u|^2 |x| \eta^2 + |\nabla u|^2 |x|^2 |\nabla (\eta^2)| \Big) \d x.
\end{split}
\end{equation}

On the other hand, testing the integral stability inequality \eqref{stab:half:jac} with $\cc = x \cdot \nabla u$,
we deduce the upper bound
\begin{equation}
\label{reo7}
\begin{split}
\int_{B_1^{+}}
&(x\cdot \nabla u)  \jacobi [x \cdot \nabla u] \, \eta^2 \d x \\
&\leq \int_{B_1^{+}} |x \cdot \nabla u|^2 \left|\nabla \eta-{\textstyle\frac{1}{2}} \eta A^{-1}(x) \bb(x)\right|^2_{A(x)} \d x\\
& \leq \int_{B_1^{+}} |x \cdot \nabla u|^2  |\nabla \eta|_{A(x)}^2 \d x+ C \varepsilon \int_{B_1} |\nabla u|^2 |x|^2 \big( |\nabla (\eta^2)| + \varepsilon \eta^2 \big)  \d x.
\end{split}
\end{equation}
Combining \eqref{reo6} and \eqref{reo7} and rearranging terms yields the claim.
\end{proof}

Recall our notation for the radial derivative
\[
r = |x|, \quad \quad u_r = \dfrac{x}{|x|} \cdot \nabla u.
\]
Given $\rho \in (0, 1/2]$, we consider a cut-off $\zeta \in C^{\infty}_c(B_{2 \rho})$ with $0 \leq \zeta \leq 1$, $\zeta =1$ in $B_{\rho}$, and $|\nabla \zeta| \leq C/\rho$ in $\supp |\nabla \zeta| \subset \overline{B_{2\rho}} \setminus B_{\rho}$.
For $a \geq 0$, by approximation, we may take the singular test function $\eta = r^{-a/2} \zeta$ in Lemma~\ref{lemma:bdy:radial}
(see~\cite[Lemma~5.4]{ErnetaInterior}),
which yields:
\begin{lemma}
\label{lemma:abs:rad}
Let $u \in W^{3,p}(B_1^+)$, for some $p > n$, be a stable solution of 
$- L u = f(u)$ in $B_{1}^{+}$, with $u = 0$ on $\partial^0 B_{1}^{+}$,
for some function $f \in C^1(\R)$.
Assume that $L$ satisfies conditions~\eqref{elliptic} and~\eqref{reg:coeffs} in $\Omega = B_1^{+}$,
and that
\[
\|D A\|_{L^{\infty}(B_{1}^{+})} + \|\vv\|_{L^{\infty}(B_{1}^{+})} \leq \varepsilon
\]
for some $\varepsilon > 0$.

If $0 \leq a \leq \min\{10, n\} - 2$, then
\[
\begin{split}
&(n-2-a) \int_{B_{\rho}^{+}}  r^{-a} |\nabla u|^2 \d x + \frac{a(8-a)}{4} \int_{B_{\rho}^{+}} r^{-a} u_r^2 \d x\\
& \quad \quad \quad  \leq  C \int_{B_{2\rho}^{+} \setminus B_{\rho}^{+}} r^{-a} |\nabla u|^2 \d x  +C \varepsilon \int_{B_{2\rho}^{+}} r^{2-a} |D^2 u| |\nabla u|  \d x \\
& \quad \quad \quad \quad \quad \quad  
 + C \varepsilon \int_{B_{2\rho}^{+}} \big( r^{1-a} + \varepsilon r^{2-a}\big) |\nabla u|^2 \d x.
\end{split}
\]
for all $\rho \leq 1/2$,
where $C$ is a universal constant.
\end{lemma}
\begin{proof}
The claim follows from Lemma~\ref{lemma:bdy:radial} as explained above,
the computations being the same as for the interior estimates in~\cite[Lemma~5.4]{ErnetaInterior}, hence we omit the proof.
\end{proof}

Lemma~\ref{lemma:abs:rad} now  allows us to prove Proposition~\ref{prop:rad}.

\begin{proof}[Proof of Proposition~\ref{prop:rad}]
Since $3 \leq n \leq 9$, we have that $\min\{10, n\}-2  = n-2$ and we may choose the exponent $a = n-2$ in Lemma~\ref{lemma:abs:rad}, leading us to the inequality
\begin{equation}
\label{reo8}
\begin{split}
&\frac{(n-2)(10-n)}{4} \int_{B^{+}_{\rho}} r^{2-n} u_r^2 \d x\\
& \quad \quad \quad  \leq  C \int_{B^{+}_{2\rho} \setminus B^{+}_{\rho}} r^{2-n} |\nabla u|^2 \d x  +C \varepsilon \int_{B^{+}_{2\rho}} r^{4-n} |D^2 u| |\nabla u|  \d x \\
& \quad \quad \quad \quad \quad \quad  
 + C \varepsilon \int_{B^{+}_{2\rho}} \big( r^{3-n} + \varepsilon r^{4-n}\big) |\nabla u|^2 \d x
\end{split}
\end{equation}
for $\rho \leq 1/2$.

It remains to control the weighted Hessian error in \eqref{reo8}.
For this, combining the boundary ``Hessian times the gradient'' estimate~\eqref{hessianw} in Theorem~\ref{thm:previous}
with the analogous interior estimates in \cite[Theorem~1.2]{ErnetaInterior},
by a simple scaling and covering argument 
it follows that
\begin{equation}
\label{reo9}
\int_{B_{\delta/2}^{+} \setminus B_{\delta/4}^{+}} r^{4-n} |D^2 u | |\nabla u| \d x \leq C \int_{B_{\delta}^{+} \setminus B_{\delta/8}^{+}} r^{3-n} |\nabla u|^2 \d x
\quad \text{ for all } \delta \in (0,1) \text{ and } \varepsilon \leq \varepsilon_0,
\end{equation}
where $\varepsilon_0> 0$ and $C$ are universal.
Letting $\delta = 2^{3-k} \rho$ in \eqref{reo9} and summing in $k \in \N$, we obtain
\begin{equation}
\label{reo10}
\int_{B^{+}_{2\rho}} r^{4-n} |D^2 u| |\nabla u|  \d x \leq 
C \int_{B_{4 \rho}^{+}} r^{3-n} |\nabla u|^2 \d x \quad \text{ for all } \rho \leq 1/4 \text{ and } \varepsilon \leq \varepsilon_0.
\end{equation}
Applying \eqref{reo10} in \eqref{reo8}, using that $(10-n)(n-2) > 0$, we deduce the claim.
\end{proof}

\section{In half-annuli the radial derivative controls the function in $L^1$} 
\label{section:l1rad}

Here we take advantage of the homogeneity of the equation to control the $L^1$ norm of a solution by the $L^1$ norm of its radial derivative.
This is an extension of a device due to Cabr\'{e}~\cite{CabreRadial} which provided quantitative proofs of the regularity of stable solutions for the Laplacian in flat domains.
Our proofs remain quantitative thanks to this idea.

Let $\tau \geq 1$ be a parameter close to $1$.
Given any function $v \colon B_1^+ \to \R$,
we denote its $L^{\infty}$ rescaling by $v^{\tau} := v(\tau \cdot)$.
Consider the elliptic operator $L v = a_{ij}(x) v_{ij} + \vv_i(x) v_i$,
with coefficients $a_{ij} \in C^{0,1}(\overline{B_1^{+}})$ and $\vv_i \in C^0(\overline{B_1^+})$.
For each $\tau$, we define $L^{\tau}$ to be the operator given by the rescaling
\begin{equation}
\label{op:res}
L^{\tau} v := \tau^{-2} a_{ij}^{\tau}(x) v_{ij}+ \tau^{-1} \vv_i^{\tau}(x) v_{i}.
\end{equation}
Our principal motivation for considering $L^{\tau}$ is the rescaling property $(Lv)^{\tau} = L^{\tau} v^{\tau}$.
In particular, given a solution $u \in W^{3,p}(B_1^+)$ to $- L u = f(u)$ in $B_1^+$, we have that
\begin{equation}
\label{go1}
- L^{\tau} u^{\tau} = f(u^{\tau}) \quad \text{ in } B_{1/\tau}^{+}.
\end{equation}
Notice that if $1 \leq \tau < 1 + \delta$, then this last equation is satisfied in $B^{+}_{1/(1+\delta)} \subset B_{1/\tau}^{+}$.
The proof of Proposition~\ref{prop:l1rad} involves taking a derivative of~\eqref{go1} with respect to $\tau$.

Before proving the proposition,
it is convenient to recall the following simple corollary of the Hessian and higher integrability estimates proven in our previous work~\cite{ErnetaBdy1}.
Here and throughout this section, we use the notation for half-annuli from~\cite{CabreRadial},
namely, for $\rho_2 > \rho_1 > 0$,
we let
\[
A^{+}_{\rho_1, \rho_2} := B_{\rho_2}^{+} \setminus \overline{B_{\rho_1}^{+}} = \{x \in \R^n \colon  x_n > 0, \rho_1 < |x| < \rho_2\}.
\]

\begin{corollary}[Corollary~3.3 in~\cite{ErnetaBdy1}]
\label{cor:annuli}
Let $u \in W^{3,p}(B_1^+)$, for some $p > n$,
be a nonnegative stable solution of 
$- L u = f(u)$ in $B_{1}^{+}$, with $u = 0$ on $\partial^0 B_{1}^{+}$.
Assume that $f \in C^1(\R)$ is nonnegative and nondecreasing.
Assume that $L$ satisfies conditions~\eqref{elliptic},~\eqref{reg:coeffs}, and~\eqref{reg:b} in $\Omega = B_1^{+}$, and that
\[
\|D A\|_{L^{\infty}(B_{1}^{+})} + \|\vv\|_{L^{\infty}(B_{1}^{+})} \leq \varepsilon
\]
for some (possibly large) $\varepsilon > 0$.
Let $0 < \rho_1 < \rho_2 < \rho_3 < \rho_4 \leq 1$.

Then
\[
\|\nabla u\|_{L^{2+\gamma}(A_{\rho_2, \rho_3}^{+})} \leq C_{\varepsilon, \rho_i} \|u\|_{L^1(A_{\rho_1, \rho_4}^{+})}
\]
and
\[
\|D^2 u \|_{L^1(A_{\rho_2, \rho_3}^{+})} \leq C_{\varepsilon, \rho_i} \|u\|_{L^1(A_{\rho_1, \rho_4}^{+})},
\]
where $C_{\varepsilon, \rho_i}$ is a (possibly large) constant depending only on $n$, $\elliptic$, $\bounded$, $\varepsilon$, $\rho_1$, $\rho_2$, $\rho_3$, and $\rho_4$.
\end{corollary}

\begin{proof}[Proof of Proposition~\ref{prop:l1rad}]
Considering the rescaled function $u(\frac{\cdot}{6})$, we may assume that we have a stable solution in $B_{6}^{+}$.

Let $\zeta \in C^{\infty}_c(A_{4, 5})$ be a nonnegative cut-off function with $\zeta = 1$ in $A_{4.1, 4.9}$.
We consider the function $\xi := x_n \zeta$, which satisfies
\[
\xi \geq 0 \text{ in } A_{4, 5}^{+}, \quad \xi = 0 \text{ on } \partial^0 A_{4, 5}^{+}, \quad \xi = \xi_\nu = 0 \text{ on } \partial^+ A_{4, 5}^{+}, \text{ and } \quad \xi = x_n \text{ in } A_{4.1, 4.9}^{+}.
\]
Multiplying \eqref{go1} (rescaled) by $\xi$ for each $1 \leq \tau \leq 1.1$ and integrating in $A_{4, 5}^{+}$, we have
\begin{equation}
\label{go2}
\int_{A_{4, 5}^{+}} (L^{\tau} u^{\tau}) \xi \d x = - \int_{A_{4, 5}^{+}} f(u^{\tau}) \d x \quad \text{ for all } \tau \in [1, 1.1].
\end{equation}
Differentiating \eqref{go2} with respect to $\tau$ and integrating, we also have
\begin{equation}
\label{go3}
\int_{1}^{1.1} \frac{\d}{\d \tau} \left( \int_{A_{4, 5}^{+}} (L^{\tau} u^{\tau}) \xi \d x  \right) \d \tau = \int_{A_{4, 5}^{+}} \left( f(u) - f(u^{1.1}) \right) \xi \d x.
\end{equation}

Our claim will be a consequence of this last identity.
For this, we first establish lower bounds for the left-hand side of \eqref{go3} by using that $L^{\tau} u^{\tau} \leq 0$.
Later, with the help of the stability inequality and the convexity of $f$, we obtain upper bounds of the right-hand side.
Finally, we control the remaining Hessian errors by applying Corollary~\ref{cor:annuli}.

\vspace{3mm}\noindent
\textbf{Step 1.}
\textit{We prove that }
\[
\begin{split}
\frac{\d}{\d \tau}\int_{A_{4, 5}^{+}} (L^{\tau} u^{\tau}) \xi \d x &\geq c \|u\|_{L^1(A_{4.7, 4.8}^{+})} - C \|u_r \|_{L^1(A_{4, 5.5}^{+})} \\
& \quad \quad \quad\quad \quad - C \varepsilon \| D^2 u\|_{L^1(A_{4, 5.5}^{+})}
- C \varepsilon \|\nabla u\|_{L^1(A_{4, 5.5}^{+})}
\end{split}
\]
\textit{ for all } $\tau \in [1, 1.1]$, \textit{ where $c$ and $C$ are universal constants.}

By definition~\eqref{op:res}, we have that
\begin{equation}
\label{gro1}
\int_{A_{4, 5}^{+}} (L^{\tau} u^{\tau}) \xi \d x = 
\int_{A_{4, 5}^{+}} \tau^{-2} a_{ij}^{\tau}(x) u^{\tau}_{ij} \, \xi d  x 
+ \int_{A_{4, 5}^{+} } \tau^{-1} b^{\tau}_{i}(x) u^{\tau}_{i}  \, \xi \d x
\end{equation}
for all $\tau \in [1, 1.1]$.

On the one hand, since $\frac{\d u^{\tau}}{\d \tau} = \tau^{-1} x \cdot \nabla u^{\tau}$, differentiating under the integral sign,
\begin{equation}
\label{gro2}
\begin{split}
&\frac{\d}{\d \tau} \left\{ \int_{A_{4, 5}^{+}} \tau^{-2} a_{ij}^{\tau}(x) u^{\tau}_{ij} \, \xi d  x \right\} \\
&= 
\int_{A_{4, 5}^{+}} \tau^{-3} a_{ij}^{\tau}(x) [ x \cdot \nabla u^{\tau}]_{ij} \, \xi d  x
- 2  \int_{A_{4, 5}^{+}} \tau^{-3} a_{ij}^{\tau}(x) u^{\tau}_{ij} \, \xi d  x
+\int_{A_{4, 5}^{+}} \tau^{-3} x \cdot \nabla a_{ij}^{\tau}(x) u^{\tau}_{ij} \, \xi d  x\\
&= 
\int_{A_{4, 5}^{+}} \tau^{-1} L^{\tau}[ x \cdot \nabla u^{\tau}] \, \xi d  x - 2  \int_{A_{4, 5}^{+}} \tau^{-1} (L^{\tau} u^{\tau}) \, \xi d  x
- \int_{A_{4, 5}^{+}} \tau^{-2} \vv_{i}^{\tau}(x) [ x \cdot \nabla u^{\tau}]_{i} \, \xi d  x
\\
& \quad \quad  \quad \quad 
+\int_{A_{4, 5}^{+}} \tau^{-3} x \cdot \nabla a_{ij}^{\tau}(x) u^{\tau}_{ij} \, \xi d  x
+ 2 \int_{A_{4, 5}^{+}} \tau^{-2} \vv_{i}^{\tau}(x) u^{\tau}_{i} \, \xi d  x.
\end{split}
\end{equation}
On the other hand, 
since $\supp \xi^{1/\tau} \subset A_{4 \tau ,5\tau}^{+} \subset A_{4, 5.5}^{+}$,
rescaling
\[
\begin{split}
\int_{A_{4, 5}^{+}} \tau^{-1} b^{\tau}_{i}(x) u^{\tau}_{i} \xi \d x
= \int_{A_{4, 5}^{+}} b^{\tau}_{i}(x) (u_i)^{\tau} \xi \d x
= \int_{A_{4, 5.5}^{+}} \tau^{-n} b_{i}(x) u_{i} \xi^{1/\tau} \d x
\end{split}
\]
and taking a derivative, we obtain
\begin{equation}
\label{gro3}
\begin{split}
&\frac{\d}{\d \tau} \left\{ \int_{A_{4, 5}^{+} } \tau^{-1} b^{\tau}_{i}(x) u^{\tau}_{i} \xi \d x \right\} 
= \frac{\d}{\d \tau} \left\{  \int_{A_{4, 5.5}^{+}} \tau^{-n} b_{i}(x) u_{i} \xi^{1/\tau} \d x \right\}
\\
&= -n \int_{A_{4, 5.5}^{+}} \tau^{-(n+1)} b_{i}(x) u_{i} \xi^{1/\tau} \d x  - \int_{A_{4, 5.5}^{+}} \tau^{-(n+2)} b_{i}(x) u_{i} \, x \cdot \nabla \xi^{1/\tau} \d x\\
&= -n \int_{A_{4, 5}^{+}} \tau^{-2} b_{i}^{\tau}(x) u^{\tau}_{i} \xi \d x  
- \int_{A_{4, 5}^{+}} \tau^{-3} b_{i}^{\tau}(x) u^{\tau}_{i} \, x \cdot \nabla \xi \d x,
\end{split}
\end{equation}
where in the last line we have 
rescaled back.
Thus, combining \eqref{gro2} and \eqref{gro3}, by \eqref{gro1}
we have
\[
\begin{split}
& \frac{\d}{\d \tau} \left\{ \int_{A_{4, 5}^{+}} (L^{\tau} u^{\tau}) \xi \d x \right\} \\ 
&= 
\int_{A_{4, 5}^{+}} \tau^{-1} L^{\tau}[ x \cdot \nabla u^{\tau}] \, \xi d  x - 2  \int_{A_{4, 5}^{+}} \tau^{-1} (L^{\tau} u^{\tau}) \, \xi d  x \\
& \quad \quad  \quad \quad 
+\int_{A_{4, 5}^{+}} \tau^{-3} x \cdot \nabla a_{ij}^{\tau}(x) u^{\tau}_{ij} \, \xi d  x
- \int_{A_{4, 5}^{+}} \tau^{-2} \vv_{i}^{\tau}(x) [ x \cdot \nabla u^{\tau}]_{i} \, \xi d  x 
\\
& \quad \quad  \quad \quad 
+ (2-n) \int_{A_{4, 5}^{+}} \tau^{-2} \vv_{i}^{\tau}(x) u^{\tau}_{i} \, \xi d  x
- \int_{A_{4, 5}^{+}} \tau^{-3} b_{i}^{\tau}(x) u^{\tau}_{i} \, x \cdot \nabla \xi \d x,
\end{split}
\]
and hence, by the bounds $1 \leq \tau \leq 1.1$ and $\|\nabla a_{ij}\|_{L^{\infty}} + \|\vv_i\|_{L^{\infty}} \leq \varepsilon$, it follows that
\begin{equation}
\label{gro4}
\begin{split}
\frac{\d}{\d \tau} \left\{ \int_{A_{4, 5}^{+}} (L^{\tau} u^{\tau}) \xi \d x \right\} & \geq 
\int_{A_{4, 5}^{+}} \tau^{-1} L^{\tau}[ x \cdot \nabla u^{\tau}] \, \xi d  x - 2  \int_{A_{4, 5}^{+}} \tau^{-1} (L^{\tau} u^{\tau}) \, \xi d  x \\
& \quad  \quad  \quad  \quad  - C \varepsilon \|D^2 u\|_{L^1(A_{4, 5.5}^{+})} - C \varepsilon \|\nabla u\|_{L^1(A_{4, 5.5}^{+})}.
\end{split}
\end{equation}

Next, we bound the two terms in the right-hand side of \eqref{gro4} by below.

For the first term, we write $L^{\tau} v = \div\left(\tau^{-2} A^{\tau}(x) v \right) + \tau^{-1} \bb^{\tau}(x) \cdot \nabla v$ in divergence form as in \eqref{op:nondiv}
and integrate by parts.
Recalling  that
$\xi$ and $\xi_{\nu}$ vanish on $\partial^+ A_{4, 5}^{+}$,
and $\xi$ and $x \cdot \nabla u^{\tau}$ vanish on $\partial^0 A_{4, 5}^{+}$,
integrating by parts twice, we have
\begin{equation}
\label{gro5}
\begin{split}
& \int_{A_{4, 5}^{+}} \tau^{-1} L^{\tau} [x \cdot \nabla u^{\tau}] \, \xi \d x \\
&\quad = \int_{A_{4, 5}^{+}} \tau^{-3} (x \cdot \nabla u^{\tau}) \div\big(A^{\tau}(x) \nabla \xi \big) \d x + \int_{A_{4, 5}^{+}} \tau^{-2} \bb^{\tau}(x) \cdot \nabla [x \cdot \nabla u^{\tau}] \, \xi \d x\\
& \quad  \geq - C \|u_r\|_{L^1(A_{4, 5.5}^{+})} - C \varepsilon \|D^2 u\|_{L^1(A_{4, 5.5}^{+})}- C \varepsilon \|\nabla u\|_{L^1(A_{4, 5.5}^{+})},
\end{split}
\end{equation}
where in the last line we have used that $1 \leq \tau \leq 1.1$,
as well as the uniform ellipticity and the bounds $\|\nabla a_{ij}\|_{L^{\infty}} + \|\vv_i\|_{L^{\infty}} \leq \varepsilon$.

The lower bounds for the second term in \eqref{gro4} are the most delicate.
Given $\rho_1 \in (4.1, 4.2)$ and $\rho_2 \in (4.8, 4.9)$,
we consider the solution $\varphi$ of the mixed boundary value problem
\[
\begin{cases}
- \Delta \varphi = 1 & \text{ in } A_{\rho_1, \rho_2}^{+}\\
\varphi = 0 & \text{ on } \partial^0 A_{\rho_1, \rho_2}^{+}\\
\varphi_{\nu} = 0 & \text{ on } \partial^+ A_{\rho_1, \rho_2}^{+}.
\end{cases}
\]
Notice that $\varphi \geq 0$ in $A_{\rho_1, \rho_2}^{+}$ by the maximum principle.
Moreover, we have the a priori estimates
$|\varphi| + |\nabla \varphi| \leq C$ in $A_{\rho_1, \rho_2}^{+}$,
where $C = C(n)$ is a dimensional constant,
and hence
\begin{equation}
\label{hopf:bd}
\xi \geq c \varphi \quad \text{ in } A_{\rho_1, \rho_2}^{+},
\end{equation}
for some small dimensional $c = c(n) > 0$ (for the details, see~\cite[Appendix~B]{CabreRadial}).

Using \eqref{hopf:bd} and by the nonnegativity of $-L^{\tau} u^{\tau} = f(u^{\tau}) \geq 0$,
we see that
\begin{equation}
\label{gro6}
-\int_{A_{4, 5}^{+}} \tau^{-1} (L^{\tau} u^{\tau}) \, \xi \d x \geq
c \int_{A_{\rho_1,\rho_2}^{+}} \tau^{-1}(-L^{\tau} u^{\tau}) \, \varphi \d x.
\end{equation}
Since $A(0) = \id$, we have that $\left|A^{\tau}(x) - \id \right| \leq \varepsilon \tau |x|$
and writing 
\[
L^{\tau} u^{\tau} = \tau^{-2}\Delta u^{\tau} + \tau^{-2}\tr\big( (A^{\tau}(x) - \id) D^2 u^{\tau}\big) + \tau^{-1}\vv^{\tau}(x) \cdot \nabla u^{\tau},
\]
by the bounds for $\varphi$, $\tau$, and the coefficients, the right-hand side of \eqref{gro6} can be further bounded by below as
\begin{equation}
\label{gro7}
\begin{split}
& \int_{A_{\rho_1,\rho_2}^{+}} \tau^{-1} (-L^{\tau} u^{\tau}) \, \varphi \d x\\
& \quad  \quad  \geq  \tau^{-3}\int_{A_{\rho_1,\rho_2}^{+}} (-\Delta u^{\tau}) \, \varphi \d x 
- C \varepsilon \int_{A_{\rho_1,\rho_2}^{+}} \tau^{-1} \left( |D^2 u^{\tau}| |x| + |\nabla u^{\tau}| \right)\d x\\
&  \quad  \quad  \geq c\int_{A_{\rho_1,\rho_2}^{+}} (-\Delta u^{\tau}) \, \varphi \d x 
- C \varepsilon \| D^2 u\|_{L^1(A_{4, 5.5}^{+})} - C \varepsilon \|\nabla u\|_{L^1(A_{4, 5.5}^{+})}.
\end{split}
\end{equation}

Following~\cite{CabreRadial}, we integrate by parts the Laplacian in~\eqref{gro7} as
\begin{equation}
\label{gro8}
\begin{split}
\int_{A_{\rho_1,\rho_2}^{+}} (-\Delta u^{\tau}) \, \varphi \d x 
&= \int_{A_{\rho_1,\rho_2}^{+}} u^{\tau} \d x - \int_{\partial^{+} A_{\rho_1,\rho_2}^{+}} (u^{\tau})_{\nu} \d \hcal^{n-1}\\
&\geq c \|u\|_{L^1(A_{\tau \rho_1, \tau \rho_2}^{+})} 
- \int_{\partial^{+} B_{\rho_1}^{+}} |(u^{\tau})_{r}| \d \hcal^{n-1}
- \int_{\partial^{+} B_{\rho_2}^{+}} |(u^{\tau})_{r}| \d \hcal^{n-1}\\
&\geq c \|u\|_{L^1(A_{4.7, 4.8}^{+})} 
- C\int_{\partial^{+} B_{\tau \rho_1}^{+}} |u_r| \d \hcal^{n-1}
- C\int_{\partial^{+} B_{\tau \rho_2}^{+}} |u_r| \d \hcal^{n-1},
\end{split}
\end{equation}
where in the last line we have used that 
$\tau \rho_1 \leq 1.1 \cdot 4.2 \leq 4.7$ and $\tau \rho_2 \geq 4.8 $.
Now, combining \eqref{gro6}, \eqref{gro7}, and \eqref{gro8}, we deduce
\begin{equation}
\label{gro9}
\begin{split}
- 2 \int_{A_{4, 5}^{+}} \tau^{-1} (L^{\tau} u^{\tau}) \, \xi \d x
&\geq c \|u\|_{L^1(A_{4.7, 4.8}^{+})} 
- C\int_{\partial^{+} B_{\tau \rho_1}^{+}} |u_{r}| \d \hcal^{n-1}
- C\int_{\partial^{+} B_{\tau \rho_2}^{+}} |u_{r}| \d \hcal^{n-1}\\
& \quad \quad  \quad  \quad - C \varepsilon \| D^2 u\|_{L^1(A_{4, 5.5}^{+})} - C \varepsilon \|\nabla u\|_{L^1(A_{4, 5.5}^{+})}.
\end{split}
\end{equation}
Integrating \eqref{gro9} in 
$\rho_1 \in (4.1, 4.2)$
and 
$\rho_2 \in (4.8, 4.9)$,
using that
$\tau \rho_1 \geq 4.1 \geq 4$ and $\tau \rho_2 \leq 1.1 \cdot 4.9 \leq 5.5$,
we finally obtain
\begin{equation}
\label{gro10}
\begin{split}
-2\int_{A_{4, 5}^{+}} \tau^{-1} (L^{\tau} u^{\tau}) \, \xi \d x
&\geq c \|u\|_{L^1(A_{4.7, 4.8}^{+})} 
- C \|u_r\|_{L^1(A_{4, 5.5}^{+})} \\
& \quad \quad \quad \quad - C \varepsilon \| D^2 u\|_{L^1(A_{4, 5.5}^{+})} - C \varepsilon \|\nabla u\|_{L^1(A_{4, 5.5}^{+})}.
\end{split}
\end{equation}
Applying \eqref{gro5} and \eqref{gro10} in \eqref{gro4} now yields the claim.

\vspace{3mm}\noindent
\textbf{Step 2.}
\textit{
We prove that for every $\delta \in (0, 1)$ and $\varepsilon_0 > 0$, we have}
\[
\begin{split}
&\int_{A_{4, 5}^{+}} \big(f(u) - f(u^{1.1})\big)\, \xi \d x 
\\
& \quad \quad \quad \quad \leq 
C_{\varepsilon_0} \left( \delta \|u \|_{L^{1}(A_{3, 6}^{+})} + \delta \|D^2 u\|_{L^1(A_{3.8, 5.8}^{+})} + \delta^{-1-2\frac{2+\gamma}{\gamma}} \|u_{r}\|_{L^1(A_{3,6}^{+})} \right)
\end{split}
\]
\textit{
for all $\varepsilon \leq \varepsilon_0$,
where $\gamma= \gamma(n) > 0$ and $C_{\varepsilon_{0}}$ depends only on $n$, $\elliptic$, $\bounded$, and $\varepsilon_{0}$.
}

Let $\phi \in C^{\infty}_c(A_{3.9, 5.1})$ be a nonnegative test function with $\phi = 1$ in $A_{4, 5}$.
Since $\xi = 0$ on~$\partial A_{4, 5}^{+}$
and
$u - u^{1.1} = 0$ on $\partial^0 A_{3.9, 5.1}^{+}$,
the functions 
$\xi$ and $(u - u^{1.1}) \phi$ are valid test functions in the integral stability inequality~\eqref{ineq:stable} with $\Omega = B_1^+$.

Since $f$ is nondecreasing, we have $f' \geq 0$.
By convexity $f(u) - f(u^{1.1}) \leq f'(u) (u - u^{1.1})$,
hence, multiplying by $\xi$, integrating,
and using the stability inequality~\eqref{ineq:stable} twice, we obtain
\begin{equation}
\label{gul1}
\begin{split}
&\int_{A_{4, 5}^{+}} \big(f(u) - f(u^{1.1})\big) \xi \d x \leq \int_{A_{4, 5}^{+}} f'(u) (u - u^{1.1}) \xi \d x\\
& \quad \leq \left(\int_{A_{4, 5}^{+}} f'(u) \xi^2 \d x\right)^{1/2} \left(\int_{A_{3.9, 5.1}^{+}} f'(u)\left( (u - u^{1.1}) \phi \right)^2 \d x\right)^{1/2}\\
& \quad  \leq \left(\int_{A_{4, 5}^{+}} |\nabla \xi - \textstyle\frac{1}{2}\xi A^{-1}(x)\bb(x) |_{A(x)}^2 \d x\right)^{1/2}\\
& \quad  \quad  \quad \cdot \left(\int_{A_{3.9, 5.1}^{+}} \big|\nabla\big\{(u - u^{1.1})\phi\big\} - \textstyle\frac{1}{2}(u - u^{1.1})\phi A^{-1}(x)\bb(x)\big|_{A(x)}^2 \d x\right)^{1/2}.
\end{split}
\end{equation}
Using that $\|\xi\|_{C^1} + \|\phi\|_{C^1} \leq C$ and the coefficient bounds,
from \eqref{gul1} it follows that
\begin{equation}
\label{gul2}
\begin{split}
&\int_{A_{4, 5}^{+}} \big(f(u) - f(u^{1.1})\big) \xi \d x \\
& \quad \leq C (1 + \varepsilon) \|\nabla (u - u^{1.1})\|_{L^2(A_{3.9, 5.1}^{+})} + C(1 + \varepsilon)^2 \|u - u^{1.1}\|_{L^2(A_{3.9, 5.1}^{+})}\\
& \quad \leq C_{\varepsilon_0} \|\nabla (u - u^{1.1})\|_{L^2(A_{3.9, 5.1}^{+})},
\end{split}
\end{equation}
where in the last line we have applied the Poincar\'{e} inequality for functions vanishing on the lower boundary~$\partial^0 A_{3.9, 5.1}^{+}$.

It remains to control the norm $\|\nabla (u - u^{1.1})\|_{L^2(A_{3.9, 5.1}^{+})}$ in \eqref{gul2}.
First we interpolate between $L^1$ and $L^{2+\gamma}$.
Letting $q = \frac{2(1 + \gamma)}{2+\gamma}$, we have
\begin{equation}
\label{gul3}
\begin{split}
\|\nabla (u - u^{1.1})\|_{L^2(A_{3.9, 5.1}^{+})} 
&\leq C \|\nabla u \|_{L^{2+\gamma}(A_{3.9, 5.1 \cdot 1.1})}^{1/q} \|\nabla (u - u^{1.1})\|_{L^1(A_{3.9, 5.1}^{+})}^{1/q'}.
\end{split}
\end{equation}
From \eqref{gul3},
by Corollary~\ref{cor:annuli} we deduce
\begin{equation}
\label{gul4}
\begin{split}
\|\nabla (u - u^{1.1})\|_{L^2(A_{3.9, 5.1}^{+})} &\leq C_{\varepsilon_0} \|u \|_{L^{1}(A_{3, 6}^{+})}^{1/q} \|\nabla (u - u^{1.1})\|_{L^1(A_{3.9, 5.1}^{+})}^{1/q'} \\
& \leq \delta \|u \|_{L^{1}(A_{3, 6}^{+})} + C_{\varepsilon_0}\delta^{-q'/q} \|\nabla (u - u^{1.1})\|_{L^1(A_{3.9, 5.1}^{+})}.
\end{split}
\end{equation}
By interpolation, 
we also have that
\begin{equation}
\label{gul5}
\begin{split}
&\|\nabla (u - u^{1.1})\|_{L^1(A_{3.9, 5.1}^{+})} \\
& \quad \quad  \leq C \delta^{1 + \frac{q'}{q}} \|D^2(u - u^{1.1})\|_{L^1(A_{3.9, 5.1}^{+})} + C \delta^{-1-\frac{q'}{q}} \|u - u^{1.1}\|_{L^1(A_{3.9, 5.1}^{+})}.
\end{split}
\end{equation}
Hence, applying \eqref{gul5} in \eqref{gul4}, we obtain
the following estimate for the Dirichlet energy
\begin{equation}
\label{gul6}
\begin{split}
&\|\nabla (u - u^{1.1})\|_{L^2(A_{3.9, 5.1}^{+})} \\
&\leq \delta \|u \|_{L^{1}(A_{3, 6}^{+})} + C_{\varepsilon_0} \delta \|D^2 u\|_{L^1(A_{3.9, 5.1}^{+})} + C_{\varepsilon_0} \delta^{-1-2\frac{q'}{q}} \|u - u^{1.1}\|_{L^1(A_{3.9, 5.1}^{+})}.
\end{split}
\end{equation}
Finally, since $\frac{\d}{\d \tau} u^{\tau}(x) = r u_r(\tau x)$, we have that
$u(x) - u^{1.1}(x) = - r \int_{1}^{1.1} u_r(\tau x) \d \tau$ 
and hence
\begin{equation}
\label{gul7}
\|u - u^{1.1}\|_{L^1(A_{3.9, 5.1}^{+})} 
\leq C \|u_r\|_{L^1(A^{+}_{3.9, 5.1\cdot 1.1})}
\leq C \|u_r\|_{L^1(A_{3, 6}^{+})}.
\end{equation}
Using \eqref{gul7} in \eqref{gul6}, and by \eqref{gul2} we deduce the claim.

\vspace{3mm}\noindent
\textbf{Step 3.} \textit{Conclusion.}

Combining Steps 1 and 2 in \eqref{go3}, 
for $\delta \in (0,1)$ and (say) $\varepsilon \leq 1$, we have
\begin{equation}
\label{ple1}
\begin{split}
\|u\|_{L^1(A_{4.7, 4.8}^{+})}
&\leq 
C \delta \|u \|_{L^{1}(A_{3, 6}^{+})}
+ C \left( \varepsilon + \delta \right) \|D^2 u\|_{L^1(A_{3.8, 5.8}^{+})} 
+ C \varepsilon \|\nabla u\|_{L^1(A_{4, 5.5}^{+})}\\
& \quad \quad \quad +C \delta^{-1-2\frac{2+\gamma}{\gamma}} \|u_{r}\|_{L^1(A_{3,6}^{+})}.
\end{split}
\end{equation}

Thanks to Corollary~\ref{cor:annuli} we can control the Hessian and gradient errors in \eqref{ple1} by the $L^1$ norm of the function.
Namely, we have
$\|D^2 u\|_{L^1(A_{3.8, 5.8}^{+})}  \leq C \|u\|_{L^1(A_{3, 6}^{+})}$
and by H\"{o}lder's inequality
$\|\nabla u\|_{L^1(A_{4, 5.5}^{+})} \leq C \|\nabla u\|_{L^{2+\gamma}(A_{4, 5.5}^{+})} \leq C \|u\|_{L^1(A_{3, 6}^{+})}$.
It follows that
\begin{equation}
\label{ple2}
\begin{split}
 \|u\|_{L^1(A_{4.7, 4.8}^{+})} & \leq
C \delta \|u \|_{L^{1}(A_{3, 6}^{+})} + C \delta^{-1-2\frac{2+\gamma}{\gamma}} \|u_{r}\|_{L^1(A_{3,6}^{+})} \quad \text{ for } \varepsilon \leq \delta.
\end{split}
\end{equation}
Proceeding as in~\cite{CabreRadial},
we write $u(s\sigma) = u(t \sigma) - \int_{s}^{t} u_r(\rho \sigma) \d \rho$
for $s \in (3, 6)$, $t \in (4.7, 4.8)$, and $\sigma \in \mathbb{S}^{n-1}$. 
Multiplying by $s^{n-1}$ and integrating in $\sigma \in \partial^+ B_1^+$, we see that
\[
\begin{split}
\int_{\partial^+ B^+_s} |u| d \mathcal{H}^{n-1} &\leq (s/t)^{n-1}\int_{\partial^{+} B^{+}_t} |u| d \mathcal{H}^{n-1}
+ \int_{3}^{6} (s/\rho)^{n-1}\int_{\partial^{+} B^{+}_{\rho}} |u_r| d \mathcal{H}^{n-1} d \rho\\
& \leq 2^{n-1} \left( 
\int_{\partial^{+} B^{+}_t} |u| d \mathcal{H}^{n-1} + 
\int_{A_{3,6}^{+}} |u_r| d x,
\right),
\end{split}
\]
where in the last term we have used coarea formula.
Integrating in $\int_{3}^{6} d s \frac{1}{0.1} \int_{4.7}^{4.8} d t$,
again by coarea formula, we conclude
\begin{equation}
\label{ple3}
\|u\|_{L^1(A_{3, 6}^{+})} \leq C \left( \|u\|_{L^1(A_{4.7, 4.8}^{+})} + \|u_r\|_{L^1(A^{+}_{3,6})} \right).
\end{equation}
Combining \eqref{ple3} and \eqref{ple2}, we deduce
\[
 \|u\|_{L^1(A_{3, 6}^{+})} \leq
C \delta \|u \|_{L^{1}(A_{3, 6}^{+})} + C \delta^{-1-2\frac{2+\gamma}{\gamma}} \|u_{r}\|_{L^1(A_{3, 6}^{+})} \quad \text{ for } \varepsilon \leq \delta,
\]
and choosing $\delta > 0$ universal small in this last inequality, we can absorb the $L^1$ norm of $u$ into the left-hand side, concluding the proof.
\end{proof}

\section{Boundary $C^\alpha$ estimate}
\label{section:holder}

We prove the H\"{o}lder estimate in half-balls.
The proof amounts to combining Propositions~\ref{prop:rad} and~\ref{prop:l1rad} with Theorem~\ref{thm:previous} to deduce the decay of the weighted Dirichlet energy, and to then apply a scaling an covering argument.


\begin{proof}[Proof of the H\"{o}lder estimate \eqref{holder} in Theorem~\ref{thm:holder}]
We may assume that $3 \leq n \leq 9$.
Indeed, when $n = 2$,
we recover the estimate by applying Theorem~\ref{thm:holder} to the function 
$\widetilde{u}(x_1, x_2, x_3) := u(x_2, x_3)$,
a stable solution to the elliptic equation $\elliptic \widetilde{u}_{x_1 x_1} + L \widetilde{u} = f(\widetilde{u})$ in $B^+_1\subset \R^{3}$,
where $L$ acts only in the $(x_2, x_3)$ variables.
Similarly, when $n = 1$, one considers the function $\widetilde{u}(x_1, x_2, x_3) := u(x_3)$.

Throughout the proof, $C$ denotes a generic universal constant unless stated otherwise.
The proof is divided into three steps.
In Step~1, we prove the weighted Dirichlet energy decay under the assumption $A(0) = \id$.
Later, in Step~2, we remove this assumption and prove a $C^{\alpha}$ estimate in universally small balls.
Finally, in Step~3, we deduce the theorem by a scaling an covering argument.

\vspace{3mm}\noindent
\textbf{Step 1:}
{\it 
Under the assumption that
\[
A(0) = \id \quad \text{ and } \quad 
\|D A\|_{L^{\infty}(B_{1}^{+})} + \|\vv\|_{L^{\infty}(B_{1}^{+})}
\leq \varepsilon,
\]
we prove that if $\varepsilon \leq \varepsilon_0$, then
\begin{equation}
\label{thedecay}
\int_{B_{\rho}^{+}} r^{2-n} |\nabla u|^2 \d x \leq C \|u\|^2_{L^1(B_1^{+})} \rho^{2\alpha} \quad \text{ for all } \rho \leq 1/16,
\end{equation}
where $\varepsilon_0 > 0$, $\alpha > 0$, and $C$ are universal constants.
}

First we write the weighted Dirichlet integral as an infinite sum on dyadic annuli, applying Corollary~\ref{cor:annuli} on each annulus.
We treat the case $\rho = 1/2$ and recover the result for general $\rho$ by rescaling.
This is the same approach used for the interior estimates in~\cite{ErnetaInterior}.

Letting $r_j := 2^{-j}$ for $j \geq 0$, we have that
\begin{equation}
\label{ann1}
\begin{split}
\int_{B_{1/2}^{+}} r^{2-n} |\nabla u|^2 \d x &
= \sum_{j = 0}^{\infty} \int_{A^{+}_{r_{j+2}, r_{j+1}}} r^{2-n} |\nabla u|^2 \d x
\leq C \sum_{j = 0}^{\infty} r_j^{2-n} \int_{A^{+}_{r_{j+2}, r_{j+1}}} |\nabla u|^2 \d x.
\end{split}
\end{equation}
We control each of the summands in \eqref{ann1}.
Combining Corollary~\ref{cor:annuli} and Proposition~\ref{prop:l1rad} applied to the functions $u(r_{j} \cdot)$,
and by H\"{o}lder inequality, it follows that
\begin{equation}
\label{almann}
r_j^{2-n} \int_{A^{+}_{r_{j+2}, r_{j+1}}} |\nabla u|^2 \d x 
 \leq C r_j^{2-n} \int_{A^{+}_{r_{j+3}, r_{j}}} u_r^2 \d x \quad \text{ for } \varepsilon\leq \varepsilon_0.
\end{equation}
Therefore, using \eqref{almann} in \eqref{ann1}, we deduce the estimate
\begin{equation}
\label{ann2}
\begin{split}
\int_{B_{1/2}^{+}} r^{2-n} |\nabla u|^2 \d x
& \leq C \sum_{j = 0}^{\infty} r_j^{2-n} \int_{A^{+}_{r_{j+3}, r_{j}}} u_r^2 \d x \leq C \int_{B_1^{+}} r^{2-n} u_r^2 \d x
\quad \text{ for } \varepsilon\leq \varepsilon_0.
\end{split}
\end{equation}

Applying \eqref{ann2} to the functions $u(2\rho \cdot)$, we obtain
\[
\begin{split}
&\int_{B_{\rho}^{+}} r^{2-n} |\nabla u|^2 \d x \leq C \int_{B_{2\rho}^{+}} r^{2-n} u_r^2 \d x \quad \text{ for all } \rho \leq 1/2 \text{ and } \varepsilon \leq \varepsilon_0,
\end{split}
\]
and by Proposition~\ref{prop:rad} (with $2 \rho$ in place of $\rho$), we conclude
\begin{equation}
\label{ann4}
\begin{split}
\int_{B^{+}_{\rho}} r^{2-n} |\nabla u|^2 \d x &\leq C \int_{B^{+}_{4\rho} \setminus B^{+}_{2\rho}} r^{2-n} |\nabla u|^2 \d x + C \varepsilon \int_{B^{+}_{8\rho}} r^{3-n}|\nabla u|^2 \d x\\
& \quad \quad \quad \quad \quad \quad \quad \quad \quad \quad \quad \quad \quad \quad \quad \quad  \text{ for all } \rho \leq 1/8 \text{ and } \varepsilon \leq \varepsilon_0.
\end{split}
\end{equation}
Splitting the last integral into $B^{+}_{8\rho} = (B^{+}_{8\rho} \setminus B^{+}_{\rho}) \cup B^{+}_{\rho}$,
since $r^{3-n} \leq r^{2-n}$ in $B_1^+$ and using that the bounds $\rho \leq 1/8$ and $\varepsilon \leq \varepsilon_0$ are universal, 
from \eqref{ann4} we deduce
\[
\begin{split}
\int_{B^{+}_{\rho}} r^{2-n} |\nabla u|^2 \d x 
&\leq C \int_{B^{+}_{8\rho} \setminus B^{+}_{\rho}} r^{2-n} |\nabla u|^2 \d x + C \varepsilon \int_{B^{+}_{\rho}} r^{2-n}|\nabla u|^2 \d x\\
& \quad \quad \quad \quad \quad \quad \quad \quad \quad \quad \quad \quad \quad \quad  \text{ for all } \rho \leq 1/8 \text{ and } \varepsilon \leq \varepsilon_0.
\end{split}
\]
Taking $\varepsilon_0 > 0$ universal smaller if necessary, 
we can absorb the last integral into the left-hand side, which yields
\begin{equation}
\label{ann6}
\begin{split}
\int_{B^{+}_{\rho}} r^{2-n} |\nabla u|^2 \d x 
&\leq C \int_{B^{+}_{8\rho} \setminus B^{+}_{\rho}} r^{2-n} |\nabla u|^2 \d x \quad \text{ for all } \rho \leq 1/8 \text{ and } \varepsilon \leq \varepsilon_0.
\end{split}
\end{equation}
Hole-filling \eqref{ann6}, it follows that
\begin{equation}
\label{ann7}
\begin{split}
\int_{B^{+}_{\rho}} r^{2-n} |\nabla u|^2 \d x 
&\leq \theta \int_{B^{+}_{8\rho}} r^{2-n} |\nabla u|^2 \d x \quad \text{ for all } \rho \leq 1/8 \text{ and } \varepsilon \leq \varepsilon_0,
\end{split}
\end{equation}
where $\theta = \frac{C}{1 + C} \in (0, 1)$ is universal.
Iterating \eqref{ann7}, for $8^{-(k+1)} < \rho \leq 8^{-k}$ we deduce
\[
\begin{split}
\int_{B^{+}_{\rho}} r^{2-n} |\nabla u|^2 \d x 
&\leq \theta^k \int_{B^{+}_{8^k\rho}} r^{2-n} |\nabla u|^2 \d x \leq \frac{1}{\theta} \rho^{2\alpha} \int_{B^{+}_1} r^{2-n} |\nabla u|^2 \d x,
\end{split}
\]
with $\alpha = -\frac{1}{2} \log_{8} \theta > 0$, and therefore
\begin{equation}
\label{ann8}
\int_{B^{+}_{\rho}} r^{2-n} |\nabla u|^2 \d x \leq C \rho^{2 \alpha}\int_{B^{+}_1} r^{2-n} |\nabla u|^2 \d x \quad \text{ for all } \rho \leq 1/8 \text{ and } \varepsilon \leq \varepsilon_0.
\end{equation}
Splitting the integral in the right-hand side of~\eqref{ann8} into $B^{+}_1 = (B^{+}_1 \setminus B^{+}_{1/8}) \cup B^{+}_{1/8}$
and applying~\eqref{ann6} with $\rho = 1/8$ to bound the solid integral in $B^+_{1/8}$ results in the bound
\begin{equation}
\label{ann9}
\int_{B^{+}_{\rho}} r^{2-n} |\nabla u|^2 \d x \leq C \rho^{2 \alpha}\|\nabla u\|_{L^2(B_1^{+})}^2 \quad \text{ for all } \rho \leq 1/8 \text{ and } \varepsilon \leq \varepsilon_0.
\end{equation}
Applying 
the energy estimate~\eqref{higherint} in Theorem~\ref{thm:previous} (rescaled) to \eqref{ann9} now yields the claim.

\vspace{3mm}\noindent
\textbf{Step 2:}
{\it 
Assuming
\[
\|D A\|_{L^{\infty}(B_{1}^{+})} + \|\vv\|_{L^{\infty}(B_{1}^{+})} \leq \varepsilon,
\]
we prove that if $\varepsilon \leq \varepsilon_0$, then 
\[
\|u\|_{C^{\alpha}(\overline{B^{+}_{\rho_{0}}})} \leq C \| u\|_{L^1(B^{+}_1)},
\]
where $\varepsilon_0 > 0$, $\alpha > 0$, $\rho_{0} > 0$, and $C$ are universal.
}

For each $y \in \partial^0 B_1^+$ and each half-ball $B^+_{d}(y) \subset B_1^{+}$, 
we can find a rotation matrix $R = R(y) \in SO(n)$ such that the function
$u^{y, d}(x) := u\big(y + \frac{d}{\sqrt{\bounded}} A^{1/2}(y) R  x\big)$
satisfies $u = 0$ on $\partial^0 B_1^{+}$,
and is a stable solution to a semilinear equation in $B_1^{+}$ with coefficients
\[
\textstyle
A^{y, d}(x) := R^{T} A^{-1/2}(y) A\big(y + \frac{d}{\sqrt{\bounded}} A^{1/2}(y) Rx \big)A^{-1/2}(y) R,
\]
\[
\textstyle
\vv^{y,d}(x) := \frac{d}{\sqrt{\bounded}}R^T A^{-1/2}(y) \vv\big(y + \frac{d}{\sqrt{\bounded}} A^{1/2}(y) R x \big).
\]
The new matrix is uniformly elliptic $\frac{\elliptic}{\bounded} \leq A^{y,d} \leq \frac{\bounded}{\elliptic}$ with $A^{y, d}(0) = \id$, 
and the coefficients can be bounded by
$\|D A^{y,d}\|_{L^{\infty}(B_{1}^{+})} + \|\vv^{y,d}\|_{L^{\infty}(B_{1}^{+})} \leq C d \left( \|D A\|_{L^{\infty}(B_{1}^{+})} + \|\vv\|_{L^{\infty}(B_{1}^{+})} 
\right)$.
Choosing $d > 0$ universal sufficiently small so that $C d \leq 1$, we further have
\[
\|D A^{y, d}\|_{L^{\infty}(B_1^+)} + \|\vv^{y, d}\|_{L^{\infty}(B_1^+)} \leq \varepsilon \quad \text{ for all } y \in \partial^0 B^{+}_{1-d}.
\]
Since $\varepsilon \leq \varepsilon_0$ (with $\varepsilon_0 > 0$ as in Step 1), by \eqref{thedecay} it follows that
\[
\int_{B^{+}_{\rho}}r^{2-n}|\nabla u^{y, d}|^2 \d x \leq C \| u^{y, d}\|_{L^1(B^{+}_1)}^2 \rho^{2\alpha} \quad \text{ for } y \in \partial^0 B^{+}_{1-d} \, \text{ and } \, \rho \leq 1/8,
\]
and using that
$r^{2-n} \geq \rho^{2-n}$ in $B_{\rho}^{+}$,
we also have that
\begin{equation}
\label{badalm}
\int_{B^{+}_{\rho}}|\nabla u^{y, d}|^2 \d x \leq C \| u^{y, d}\|_{L^1(B^{+}_1)}^2 \rho^{2\alpha+n-2} \quad \text{ for } \, y \in \partial^0 B^{+}_{1-d} \, \text{ and }\,  \rho \leq 1/8.
\end{equation}

We now write
\eqref{badalm} in terms of the original function $u$.
By the change of variables $z = y + \frac{d}{\sqrt{\bounded}} A^{1/2}(y) R x$
and by uniform ellipticity, 
using that $B^{+}_{\sqrt{\elliptic }\rho} \subset A^{1/2}(y) R (B^{+}_{\rho})$,
the left-hand side of \eqref{badalm} can be bounded from below by
\begin{equation}
\label{badalm1}
\begin{split}
\int_{B_{\rho}^{+}} |\nabla u^{y, d}|^2 \d x 
&\geq c \, d^{2-n} \int_{B^{+}_{d \sqrt{\frac{\elliptic}{\bounded}}\rho}(y)} |\nabla u|^2 \d z,
\end{split}
\end{equation}
where 
$c > 0$ is a universal constant.
Similarly, we have
$\|u^{y,d}\|_{L^1(B_1^{+})} \leq C d^{-n} \|u\|_{L^1(B_1^{+})}$
and from \eqref{badalm} and \eqref{badalm1} we deduce
\begin{equation}
\label{badalm2}
\int_{B^{+}_{d \sqrt{\frac{\elliptic}{\bounded}}\rho}(y)} |\nabla u|^2 \d z \leq C d^{-2} \|u\|_{L^1(B^{+}_1)}^2 \rho^{n-2+2\alpha} \quad \text{ for }  
y \in \partial^0 B^{+}_{1-d}
\, \text{ and } \,  \rho \leq 1/8.
\end{equation}
Let $\rho_{0} := \frac{d}{16} \sqrt{\frac{\elliptic}{\bounded}}$.
Taking $d$ smaller if necessary, we may assume that $B^{+}_{2\rho_{0}} \subset B^{+}_{1-d}$.
Dividing $\rho$ by $d \sqrt{\frac{\elliptic}{\bounded}}$ in \eqref{badalm2}, using that $d$ is universal, and by Cauchy--Schwarz, we obtain
\begin{equation}
\label{badalm4}
\int_{B^{+}_{\rho}(y)} |\nabla u| \d z \leq C \| u\|_{L^1(B^{+}_1)} \rho^{n-1+\alpha} 
\quad \text{ for }  
y \in \partial^0 B^{+}_{2\rho_{0}}
\, \text{ and } \,  \rho \leq 2\rho_{0}.
\end{equation}

With \eqref{badalm4} in hand, we are finally ready to prove the boundary H\"{o}lder estimate.

Let $x = (x', x_n) \in B_{\rho_{0}}^{+} \subset \R^n \times \R_+$.
Since $u = 0$ on $\partial^0 B_1^{+}$,
by the Poincar\'{e} inequality
\begin{equation}
\label{poingo}
\|u\|_{L^1(B^{+}_{2 x_n}(x',0))} \leq C x_n \|\nabla u\|_{L^1(B^{+}_{2 x_n}(x',0))}.
\end{equation}
Applying \eqref{badalm4} with $\rho = 2x_n$ and $y = (x', 0)$, from \eqref{poingo} we deduce
\begin{equation}
\label{triv}
\|u\|_{L^1(B^{+}_{2 x_n}(x',0))} 
\leq C \|u\|_{L^1(B_1^{+})} (x_n)^{n+\alpha}.
\end{equation}
By the interior H\"{o}lder estimates in \cite[Theorem~1.1]{ErnetaInterior},
in the ball $B_{x_n}(x) \subset B_1^{+}$ 
we have
\begin{equation}
\label{triv3}
\|u\|_{L^{\infty}(B_{x_n/2}(x))} + (x_n)^{\alpha} [u]_{C^{\alpha}(\overline{B}_{x_n/2}(x))} 
\leq C \|u\|_{L^1(B_{x_n}(x))} (x_n)^{-n},
\end{equation}
where $\alpha > 0$ and $C$ are universal constants (since we are assuming a universal bound $\varepsilon \leq \varepsilon_0$ on the coefficients).
Since $B_{x_n}(x) \subset B^{+}_{2x_n}(x', 0)$, combining \eqref{triv3} and \eqref{triv} gives
\begin{equation}
\label{triv2}
\|u\|_{L^{\infty}(B_{x_n/2}(x))} + (x_n)^{\alpha} [u]_{C^{\alpha}(\overline{B}_{x_n/2}(x))} 
\leq C \|u\|_{L^1(B_1^{+})} (x_n)^{\alpha} \quad \text{ for } x \in B_{\rho_{0}}^{+}.
\end{equation}

In particular, from \eqref{triv2} it follows that
$|u(x)| \leq C \|u\|_{L^1(B_1^{+})} (x_n)^{\alpha}$ in $B_{\rho_{0}}^{+}$, and we have controlled the $L^{\infty}$ norm of $u$ in $B_{\rho_{0}}^{+}$.
To bound the H\"{o}lder norm in $B_{\rho_{0}}^{+}$, consider $x, y \in B_{\rho_{0}}^{+}$ such that $x \neq y$.
Without loss of generality 
we may assume $y_n \leq x_n$.
On the one hand,
if $|x- y| \leq x_n/2$, then from \eqref{triv2} we deduce
\[
\frac{|u(x)-u(y)|}{|x-y|^{\alpha}} \leq [u]_{C^{\alpha}(\overline{B}_{x_n/2}(x))} \leq C \|u\|_{L^1(B_1^{+})}.
\]
On the other hand,
if $|x- y| > x_n/2$, then by the $L^{\infty}$ estimates in \eqref{triv2} it follows that
\[
\begin{split}
|u(x) - u(y)| &\leq u(x) + u(y) \leq C \|u\|_{L^1(B_1^{+})} \left( (x_n)^{\alpha} + (y_n)^{\alpha} \right) \leq C \|u\|_{L^1(B_1^{+})} (x_n)^{\alpha}\\
&\leq C \|u\|_{L^1(B_1^{+})}|x-y|^{\alpha}.
\end{split}
\]
Combined, the two inequalities above yield a bound for $[u]_{C^{\alpha}(B_{\rho_{0}}^{+})}$, which was the claim.

\vspace{3mm}\noindent
\textbf{Step 3:}
{\it Conclusion. }

Arguing as in the proof of the higher integrability estimate in our previous paper~\cite{ErnetaBdy1},
by a scaling and covering argument
and using the interior estimates in \cite[Theorem~1.1]{ErnetaInterior},
is is not hard to deduce the theorem in its final form from Step 2.

\end{proof}

\section{Approximation and proof in $C^{1,1}$ domains}
\label{section:approximation}

Here we give the complete proof of our main result, Theorem~\ref{thm:c11}, which establishes a priori estimates in $C^{1,1}$ domains. 
By an approximation argument (carried out in the proof at the end of this section), it will suffice to obtain these estimates in smooth domains.
%
%
\medskip

First, we comment on the invariance of our class of solutions under general transformations flattening the boundary.
This fact will allow us to reduce the problem in smooth domains to half-balls, where we already obtained a priori estimates in Theorem~\ref{thm:holder} above.
\medskip

Given a smooth bounded domain $\Omega \subset \R^n$, we can write it as the superlevel set of a smooth function $\Phi \in C^{\infty}(\R^n)$
with $\nabla \Phi \neq 0$ on $\partial \Omega$ and 
\[
\Omega = \{x \in \R^n \colon \Phi(x) > 0\} = \{\Phi > 0\}.
\]
For each $x_0 \in \partial \Omega$, using $\Phi$ we can construct a smooth map $\Psi = \Psi_{x_0} \colon \R^n \to \R^n$ which flattens out the boundary $\partial \Omega$ in a neighborhood of $x_0$.
Assuming that $\nabla \Phi(x_0) = |\nabla \Phi(x_0)| e_n$ by rotation, and writing $x = (x', x_n) \in \R^{n-1}\times \R$, we can take $\Psi(x) = ( (x-x_0)', \frac{\Phi(x)}{|\nabla \Phi(x_0)|})$.
This map is a diffeomorphism in a ball $B_{R_2}(x_0) \subset \R^n$ such that
\[
\Psi(B_{R_2}(x_0) \cap \Omega) \subset \R^n_{+} \quad \text{ and } \quad \Psi(B_{R_2}(x_0) \cap \partial \Omega) \subset \partial \R^n_{+},
\]
and satisfies the inclusions
\begin{equation}
\label{eq:lmao}
\Psi(B_{R_1}(x_0) \cap \Omega) 
= B_{\rho/2}^{+} \cap \Psi(B_{R_1}(x_0)),
\quad \text{ and } \quad 
\Psi(B_{R_2}(x_0) \cap \Omega) \cap B_{\rho} = B_{\rho}^{+},
\end{equation}
for some small numbers 
$0 < R_1 < R_2$ and $\rho > 0$
depending only on  $\|\nabla \Phi\|_{C^{0,1}(\R^n)}$ and $\||\nabla \Phi|^{-1}\|_{L^{\infty}(\partial \Omega)}$, but not on the point $x_0$.
We give the details in Appendix~\ref{app:approximation} (see Lemma~\ref{lemma:prelim} and the discussion preceding it).

Consider now the differential operator $L u(x) = a_{ij}(x) u_{ij} + b_i(x) u_i$ 
acting on functions $u$ in $C^2(B_{R_2}(x_0) \cap \overline{\Omega})$.
In the new coordinates $\widetilde{x} = \Psi(x)$, the function $\widetilde{u} = u \circ \Psi^{-1}$ satisfies
\[
\widetilde{L} \widetilde{u} := (L u )(\Psi^{-1}(\widetilde{x})) = \widetilde{a}_{ij}(\widetilde{x}) \widetilde{u}_{ij} + \widetilde{b}_{i}(\widetilde{x}) \widetilde{u}_{i},
\]
where the new coefficients are given by
\[
\widetilde{a}_{ij} \circ \Psi(x) = 
a_{kl}(x) \partial_k \Psi_i(x) \partial_l \Psi_{j}(x) 
\]
and
\[
\widetilde{b}_{i} \circ \Psi(x) = 
b_{k}(x) \partial_k \Psi_i(x) + a_{jk}(x) \partial^2_{jk} \Psi_i(x).
\]

When $0 < \elliptic \leq A(x) \leq \bounded$, taking $R_2 = R_2 (\|\nabla \Phi\|_{C^{0,1}(\R^n)}, \||\nabla \Phi|^{-1}\|_{L^{\infty}(\partial \Omega)})> 0$ 
smaller if necessary,
we may assume that the new coefficient matrix $\widetilde{A}(\widetilde{x}) = (\widetilde{a}_{ij}(\widetilde{x}))$ is also uniformly elliptic with (say)
\[
0 < 
\frac{1}{2} \elliptic \leq \widetilde{A}(\widetilde{x}) \leq \frac{3}{2} \bounded.\footnote{Indeed, since $\widetilde{A}(\widetilde{x}) = D \Psi(x)^T A(x) D \Psi(x)$, by ellipticity $\elliptic |D \Psi(x) p |^2 \leq \widetilde{A}(\widetilde{x}) p \cdot p \leq \bounded |D \Psi(x) p|^2$. Moreover, since $|D \Psi(x) p|^2 = |p'|^2 + |\frac{\nabla' \Phi(x)}{\partial_n \Phi(x_0) }\cdot p' + \frac{\partial_n \Phi(x)}{\partial_n \Phi(x_0)}  p_n|^2$, choosing $R_2 > 0$ smaller such that $\frac{9}{10} \leq \frac{\partial_n \Phi(x)}{\partial_n \Phi(x_0)} \leq \frac{|\nabla \Phi(x)|}{\partial \Phi(x_0)} \leq \frac{11}{10}$ and $\frac{|\nabla' \Phi(x)|}{\partial \Phi(x_0)} \leq \frac{1}{100}$, it is easy to check that $\frac{1}{2} |p|^2 \leq |D \Psi(x) p|^2 \leq \frac{3}{2} |p|^2$ and the claim follows.}
\]

If $u$ is a stable solution of $-L u = f(u)$ in $ B_{R_2}(x_0) \cap \Omega$,
then $\widetilde{u}$ is a stable solution of $- \widetilde{L} \widetilde{u} = f(\widetilde{u})$ in $\Psi(B_{R_2}(x_0) \cap \Omega) \subset \R^n_{+}$.
To see that $\widetilde{u}$ is stable, notice that if $\varphi > 0$ is the function in the definition of stability~\eqref{stable:point},
then $\widetilde{\varphi} = \varphi \circ \Psi^{-1}$ satisfies the same condition with respect to $\widetilde{L}+ f(\widetilde{u})$.

Finally, notice that the new coefficient norms
$\| D \widetilde{A}\|_{L^{\infty}}$ and $\|\widetilde{\vv} \|_{L^{\infty}}$ involve the norms $\| D A\|_{L^{\infty}}$, $\|\vv \|_{L^{\infty}}$, $\|D \Phi\|_{L^{\infty}}$, and $\|D^2 \Phi\|_{L^{\infty}}$.
This dependence will be crucial to extend our results to $C^{1,1}$ domains, which are described by a $C^{1,1} = W^{2, \infty}$ function $\Phi$.

Thanks to all these preliminaries, we are now in a position to upgrade our estimates from half-balls to hold in $C^{1,1}$ domains:

\begin{proof}[Proof of Theorem~\ref{thm:c11}]

We proceed in two steps.
First, we prove the theorem in smooth domains, for solutions in $W^{3,p}(\Omega)$,
i.e., 
with weak derivatives that are integrable \emph{up to the boundary}.
Then, we approximate our $C^{1,1}$ domain from the interior by smooth domains and apply the first step on a suitable sequence of stable solutions.

As explained in Appendix~\ref{app:approximation},
for $\Omega \subset \R^n$ of class $C^{1, 1}$,
there is a function $\Phi \in C^{1,1}(\R^n)$ such that $\Omega = \{\Phi > 0\}$ and $\nabla \Phi \neq 0$ on $\partial \Omega$.
When $\Omega$ is smooth, $\Phi$ can be chosen to be in $C^{\infty}$.
The main purpose of the function $\Phi$ is to quantify the dependence of our bounds on the domain.
In the proof below, we denote the diameter of $\Omega$ by $\diam(\Omega) := \sup_{x, y \in \Omega} |x-y|$.

\vspace{3mm}
\noindent
\textbf{Step 1:}
{\it 
Under the additional assumptions that $\Omega$ is smooth, 
$L$ satisfies \eqref{reg:b}, and $u \in W^{3,p}(\Omega)$ for some $p > n$,
we prove that
\[
\|\nabla u\|_{L^{2+ \gamma}(\Omega)} \leq C \|u\|_{L^1(\Omega)},
\]
where $\gamma > 0$ is dimensional and
$C$ is a constant depending only on 
$n$, $\elliptic$, 
$\bounded$, $\|D A\|_{L^{\infty}(\Omega)}$, 
$\|\vv\|_{L^{\infty}(\Omega)}$, $\|\nabla \Phi\|_{C^{0,1}(\R^n)}$, $\||\nabla \Phi|^{-1}\|_{L^{\infty}(\partial \Omega)}$, and 
$\diam(\Omega)$.
In addition,
\[
\|u\|_{C^{\alpha}(\overline{\Omega})} \leq C \|u\|_{L^1(\Omega)} \quad \text{ if } n \leq 9,
\]
where $\alpha > 0$ is universal and $C$ is a constant depending only on $n$, $\elliptic$, $\bounded$, $\|D A\|_{L^{\infty}(\Omega)}$, $\|\vv\|_{L^{\infty}(\Omega)}$, $\|\nabla \Phi\|_{C^{0,1}(\R^n)}$, and $\||\nabla \Phi|^{-1}\|_{L^{\infty}(\partial \Omega)}$.
}
\medskip

Let $0 < R_1 < R_2$ be the functions of $\|\nabla \Phi\|_{C^{0,1}(\R^n)}$ and $\||\nabla \Phi|^{-1}\|_{L^{\infty}(\partial \Omega)}$ from~\eqref{eq:lmao}.

Let $\delta := R_1/3 > 0$. 
Since $\Omega$ is bounded, it is contained in a ball of radius ${\rm diam}(\Omega) < \infty$.
Hence, we can cover $\overline{\Omega}$ by $N$ balls $\{B_i\}_{i}$ of radius $\delta$, where $N \leq C \diam(\Omega)^n \delta^{-n}$ for some dimensional $C$.\footnote{To see this, given $k \in \N$ and $R > 0$, consider the set $A_{k, R} = \{R l /k \colon l \in \Z \text{ with }  -k \leq l \leq k\}^n \subset \R^n$. Notice that $A_{k,R}$ is a discrete set with $N = (2k+1)^{n}$ elements. For each $x \in B_R$, there is a $y \in A_{k, R}$ such that $|x - y| \leq \sqrt{n} \frac{R}{2k}$. Hence, if $\sqrt{n} \frac{R}{2 \delta} < k \leq \sqrt{n} \frac{R}{2 \delta} + 1$, then $B_{R} \subset \cup_{y \in A_{k, R}} B_{\delta}(y)$ and the number of balls can be estimated by $N \leq (\sqrt{n} \frac{R}{\delta} + 3)^n \leq (2 \sqrt{n})^n \left(\frac{R}{\delta}\right)^{n}$ by taking $\delta$ smaller in terms of $R$ and $n$.}
We write $2 B_i$ 
to denote the ball centered at the same point but with twice the radius.
We label the balls $B_i$ in a way such that the first 
$N' < N$ are close to the boundary, 
in the sense that $2 B_i \cap \partial \Omega \neq \emptyset$,
and the remaining $N- N'$ are interior, i.e., they satisfy the inclusion $2 B_i \subset \Omega$.

For each $i \leq N'$, by definition, there is a boundary point $x_i \in 2 B_i \cap \partial \Omega$ 
and hence 
$B_i \subset B_{3 \delta}(x_i) = B_{R_1}(x_i)$.
In particular, flattening the boundary as explained above and applying the energy estimate~\eqref{higherint} in Theorem~\ref{thm:holder} (rescaled), we deduce
\begin{equation}
\label{high:bdyballs}
\|\nabla u\|_{L^{2+\gamma}(B_i \cap \Omega)} \leq \|\nabla u\|_{L^{2+\gamma}(B_{R_1}(x_i) \cap \Omega)} \leq 
C \|u\|_{L^1(B_{R_2}(x_i) \cap \Omega)} \quad \text{ for all } i \leq N',
\end{equation}
where $C = C(n, \elliptic, \bounded, \|D A\|_{L^{\infty}(\Omega)}, \|\vv\|_{L^{\infty}(\Omega)}, \| \nabla \Phi\|_{C^{0,1}(\R^n)}, \||\nabla \Phi|^{-1}\|_{L^{\infty}(\partial \Omega)})$.
Therefore, by \eqref{high:bdyballs} and interior estimates~\cite[Theorem~1.1]{ErnetaInterior}, we conclude
\[
\begin{split}
\|\nabla u\|_{L^{2+\gamma}(\Omega)} &\leq \sum_{i \leq N'}\|\nabla u\|_{L^{2+\gamma}(B_i \cap \Omega)} + \sum_{i > N'} \|\nabla u\|_{L^{2+\gamma}(B_i)} \\
&\leq C \sum_{i \leq N'} \|u\|_{L^1(B_{R_2}(x_i) \cap \Omega)}
+ C \sum_{i > N'}\| u\|_{L^{1}(2 B_i)} \\
& \leq C \|u\|_{L^1(\Omega)},
\end{split}
\]
where $C$ depends only on $n$, $\elliptic$, $\bounded$, $\|D A\|_{L^{\infty}(\Omega)}$, $\|\vv\|_{L^{\infty}(\Omega)}$, $\|\nabla \Phi\|_{C^{0,1}(\R^n)}$, $\||\nabla \Phi|^{-1}\|_{L^{\infty}(\partial \Omega)}$, and $\diam(\Omega)$.
The first claim follows.
\medskip

Assume now that $n \leq 9$.

We prove the $L^{\infty}$ estimate first.
Let $x \in \Omega$.
If ${\rm dist}(x, \partial \Omega) < R_1$, then
$x \in B_{R_1}(x_0) \cap \Omega$ for some $x_0 \in \partial \Omega$.
Hence, flattening the boundary,
by Theorem~\ref{thm:holder} we obtain
\[
|u(x) | \leq \|u\|_{L^{\infty}(B_{R_1}(x_0) \cap \Omega)} \leq C \|u\|_{L^1(B_{R_2} \cap \Omega)} \leq C \|u\|_{L^1(\Omega)},
\]
where 
$C = C(n, \elliptic, \bounded, \|D A\|_{L^{\infty}(\Omega)}, \|\vv\|_{L^{\infty}(\Omega)}, \| \nabla \Phi\|_{C^{0,1}(\R^n)}, \||\nabla \Phi|^{-1}\|_{L^{\infty}(\partial \Omega)})$.
Otherwise, if ${\rm dist}(x, \partial \Omega) \geq R_1$,
then by interior estimates~\cite[Theorem~1.1]{ErnetaInterior} (rescaled) we have
\[
|u(x) | \leq \|u\|_{L^{\infty}(B_{R_1/2}(x)} \leq C \|u\|_{L^1(B_{R_1}(x))} \leq C \|u\|_{L^1(\Omega)},
\]
where again $C = C(n, \elliptic, \bounded, \|D A\|_{L^{\infty}(\Omega)}, \|\vv\|_{L^{\infty}(\Omega)}, \| \nabla \Phi\|_{C^{0,1}(\R^n)}, \||\nabla \Phi|^{-1}\|_{L^{\infty}(\partial \Omega)})$.
The two inequalities yield the desired $L^{\infty}$ bound $\|u\|_{L^{\infty}(\Omega)} \leq C \|u\|_{L^1(\Omega)}$.

Next, we estimate the $C^{\alpha}$ seminorm.
Let $x, y \in \Omega$ with $x \neq y$
and let $\widetilde{\delta} := 2 R_1/3 > 0$.
We distinguish three cases:
\begin{itemize}
\item If $|x-y| \geq \widetilde{\delta}/2$, then 
\[
\frac{|u(x) - u(y)|}{|x-y|^{\alpha}} \leq 2^{\alpha}\widetilde{\delta}^{-\alpha} \left( |u(x)| + |u(y)|\right) \leq 2^{\alpha+1}\widetilde{\delta}^{-\alpha} \|u\|_{L^{\infty}(\Omega)}
\]
and we apply the $L^{\infty}$ estimate $\|u\|_{L^{\infty}(\Omega)} \leq C \|u\|_{L^{1}(\Omega)}$ to control the right-hand side.
\item If $|x- y| < \widetilde{\delta}/2$ and (say) ${\rm dist }(x, \partial \Omega) < \widetilde{\delta}$,
then $x, y \in B_{\widetilde{\delta}/2}(x) \subset B_{\frac{3}{2}\widetilde{\delta}}(x_0) \subset B_{R_1}(x_0)$ for some $x_0 \in \partial \Omega$.
Hence, flattening the boundary,
by Theorem~\ref{thm:holder} we deduce
\[
\begin{split}
\frac{|u(x) - u(y)|}{|x -y|^{\alpha}} &\leq [u]_{C^{\alpha}(\overline{B_{\widetilde{\delta}/2}(x) \cap \Omega})} \leq [u]_{C^{\alpha}(\overline{B_{R_1}(x_0) \cap \Omega})} \leq C \|u\|_{L^1(B_{R_2}(x_0) \cap \Omega)} \\
&\leq C \|u\|_{L^1(\Omega)}.
\end{split}
\]
\item If $|x- y| < \widetilde{\delta}/2$ and $\min\{ {\rm dist }(x, \partial \Omega), {\rm dist }(y, \partial \Omega) \} \geq \widetilde{\delta}$,
then we have the inclusions
$x, y \in B_{\widetilde{\delta}/2}(x) \subset B_{\widetilde{\delta}}(x) \subset \Omega$ and by interior estimates~\cite[Theorem~1.1]{ErnetaInterior} (rescaled) we deduce
\[
\frac{|u(x) - u(y)|}{|x -y|^{\alpha}} \leq [u]_{C^{\alpha}(\overline{B}_{\widetilde{\delta}/2}(x) )} \leq C \|u\|_{L^{1}(B_{\widetilde{\delta}}(x))} \leq C \|u\|_{L^1(\Omega)}.
\]
\end{itemize}
The three inequalities yield the bound
\[
[u]_{C^{\alpha}(\overline{\Omega})} \leq C \|u\|_{L^1(\Omega)},
\]
where
$C = C(n, \elliptic, \bounded, \|D A\|_{L^{\infty}(\Omega)}, \|\vv\|_{L^{\infty}(\Omega)}, \| \nabla \Phi\|_{C^{0,1}(\R^n)}, \||\nabla \Phi|^{-1}\|_{L^{\infty}(\partial \Omega)})$.
This concludes the proof of Step 1.

\vspace{3mm}
\noindent
\textbf{Step 2:}
{\it 
Conclusion: Approximation argument.
}
\medskip

Let $\Omega_k = \{\Phi_k > 0\}$ be an exhaustion of $\Omega = \{\Phi > 0\}$ by $C^{\infty}$ sets,
with norms 
satisfying
\begin{equation}
\label{bdd:norms}
\| \nabla \Phi_k\|_{C^{0,1}(\R^n)} + \||\nabla \Phi_k|^{-1}\|_{L^{\infty}(\partial \Omega_k)} \leq C
\end{equation}
for some constant $C$ depending only on $\Phi$ and $\Omega$ (for the details, see Appendix~\ref{app:approximation}).

For each $k$, let $\vv_i^k := \vv_i * \eta_k$,
where $(\eta_k)_{k}$ is a smooth regularizing sequence such that $\vv_i^k \in C^{\infty}(\overline{\Omega_k})$.
In particular, since $\vv_i \in C(\Omega)$, we have that $\vv_i^k \to \vv_i$ locally uniformly in $\Omega$.
We define the operator
\[
L_{k} := a_{ij}(x) \partial_{ij} + \vv_i^k(x) \partial_i,
\]
where $a_{ij} \in C^{0,1}(\overline{\Omega})$ is the same coefficient as in the statement of the theorem.
By elliptic regularity,
all bounded strong solutions of the problem $- L_k u_k = f(u_k)$ in $\Omega_k$, $u_k = 0$ on $\partial \Omega_k$,
belong to $W^{3,p}(\Omega_k)$.

We will distinguish the two cases $f(0) > 0$ and $f(0) = 0$.
\medskip

\textbf{Case $f(0) > 0$.}
Let $\varepsilon_k \in (0,1)$ with $\varepsilon_k \downarrow 0$. 
For each $k$, we will construct a stable solution $u_k \in W^{3,p}(\Omega_k)$ to the problem
\begin{equation}
\label{eq:k}
\left\{
\begin{array}{cl}
- L_k u_k = (1 - \varepsilon_k) f(u_k) & \text{ in } \Omega_k\\
u_k > 0 & \text{ in } \Omega_k\\
u_k = 0 & \text{ on } \partial \Omega_k
\end{array}
\right.
\end{equation}
by monotone iteration starting at $0$.
This is the so called \emph{minimal solution} of \eqref{eq:k}, i.e., the smallest positive supersolution of \eqref{eq:k},
and it is well-known to be stable.\footnote{To prove that it is stable one argues by contradiction, considering $\phi$, the principal eigenfunction of $L + (1-\varepsilon_k)f'(u_k)$ in $\Omega_k$, and showing that $u_k - \delta \phi$ would be a positive supersolution for $\delta > 0$ sufficiently small; see~\cite{BerestyckiKiselevNovikovRyzhik}.}

For the monotone iteration to converge, we need to exhibit a barrier function.
We claim that $u \in C^{0}(\overline{\Omega}) \cap W^{2, n}_{\rm loc}(\Omega)$ is a barrier for \eqref{eq:k},
in fact,
we have $u > 0$ on $\overline{\Omega_k} \subset \Omega$ and
\[
- L_k u \geq (1 - \varepsilon_k) f(u) \quad \text{ in } \Omega_k.
\]
Indeed, using the equation satisfied by $u$ and by monotonicity of $f$, we have that
\[
- L_k u - (1 - \varepsilon_k) f(u) = \varepsilon_k f(u) + (L - L_k) u \geq \varepsilon_k f(0) - \|\vv - \vv^k\|_{L^{\infty}(\Omega_k)} \|\nabla u\|_{L^{\infty}(\Omega_k)},
\]
and the right-hand side is nonnegative by choosing the regularizing sequence $\eta_k$ in terms of $\varepsilon_k$, $f(0)$, and $\|\nabla u_k\|_{L^{\infty}(\Omega_k)}$
so that $\|b - b^k\|_{L^{\infty}(\Omega_k)}$ is sufficiently small.
Here, recall that by $L^p$ estimates we have $u \in C^{0}(\overline{\Omega}) \cap W^{2, p}_{\rm loc}(\Omega)$ for all $p < \infty$, and hence $u \in C^{1, \alpha}(\overline{\Omega_k})$ for all $\alpha \in (0,1)$ and $k \in \N$.

Next, we carry out the monotone iteration.
Let $k \in \N$ and $u_k^{(0)} = 0$.
For each $l \in \N$, we consider the unique solution $u^{(l)}_{k}$ to the problem $- L_k u_k^{(l)} = (1-\varepsilon_k)f(u^{(l-1)}_k)$ in $\Omega_k$, $u^{(l)}_{k} = 0$ on $\partial \Omega_k$. 
Since $u$ is a barrier, by maximum principle, it is not hard to show that $0 \leq u_k^{(l-1)} \leq u_k^{(l)} \leq u$ in $\Omega_k$, and by global regularity the monotone limit $u_k := \lim_{l \uparrow \infty} u_k^{(l)}$ converges uniformly in $C^2(\overline{\Omega_k})$ norm and solves \eqref{eq:k}.
By construction, $u_k$ is below any supersolution of \eqref{eq:k} and hence it is the minimal solution.

Since $\overline{\Omega_{k}} \subset \Omega_{k+1} \subset \Omega$, by maximum principle we have
\[
0 \leq u_k \leq u_{k+1} \leq u \quad \text{ in } \Omega_k.
\]
Let $u^{\star}(x) := \lim_{k \to \infty} u_k(x)$ for $x \in \Omega$.
Using that $u \in C^{0}(\overline{\Omega})$, 
by $L^p$ estimates we have $u_k \to u^{\star}$ weakly in $W^{2,p}_{\rm loc}(\Omega)$ for all $p < \infty$.
Since $u^{\star} \leq u$ in $\Omega$, we can extend $u^{\star}$ up to the boundary to a function $u^{\star} \in C^{0}(\overline{\Omega}) \cap W^{2,p}_{\rm loc}(\Omega)$.
By weak convergence, it follows that $u^{\star}$ is a strong solution to $- L u^{\star} = f(u^{\star})$ in $\Omega$, $u^{\star} = 0$.
Moreover, 
$u^{\star}$ is a stable solution.
To see this, taking a positive function $\varphi \in W^{2,n}_{\rm loc}(\Omega)$ as in the stability inequality~\eqref{stable:point} for $u$, using that $u^{\star} \leq u$ and by the convexity of $f$, we have
\[
J_{u^{\star}} \varphi = (L + f'(u^{\star}))\varphi \leq (L + f'(u)) \varphi =  \jacobi \varphi \leq 0 \quad \text{ in } \Omega.
\]

Finally, by the uniqueness of stable solutions for convex nonlinearities (see Appendix~\ref{app:uniqueness}), we conclude $u = u^{\star}$.


Applying Step 1 to the minimal solutions $u_k$, 
by the bounds
$\|\vv^k\|_{L^{\infty}(\Omega_k)} \leq \|\vv\|_{L^{\infty}(\Omega)}$
and~\eqref{bdd:norms}, 
and by the monotonicity of the sequence,
for $k \geq l +1$
we obtain
\[
\|\nabla u_k\|_{L^{2+\gamma}(\Omega_l)} 
\leq \|\nabla u_k\|_{L^{2 + \gamma}(\Omega_k)} 
\leq C \|u_k\|_{L^1(\Omega_k)} 
\leq C \|u\|_{L^1(\Omega)}.
\]
Now, since $\nabla u_k \to \nabla u$ uniformly on compacts,
letting $k\to \infty$ and $l \to \infty$ in this last estimate and by monotone convergence,
we deduce the higher integrability in $C^{1,1}$ domains.

Assuming moreover that $n \leq 9$, by Step 1 applied to $u_k$,
for $k \geq l +1$ we have
\[
\| u_k\|_{C^{\alpha}(\overline{\Omega_l})} 
\leq \| u_k\|_{C^{\alpha}(\overline{\Omega_k})} 
\leq C \|u_k\|_{L^1(\Omega_k)} 
\leq C \|u^{\star}\|_{L^1(\Omega)} = C \|u\|_{L^1(\Omega)}.
\]
Hence, letting $k\to \infty$ and $l \to \infty$,
we deduce the H\"{o}lder estimate in $C^{1,1}$ domains.
\medskip

\textbf{Case $f(0) = 0$.}
Without loss of generality, we may assume that $u > 0$ in $\Omega$. 
Since $0$ is a stable solution,
by Proposition~\ref{prop:uniqueness} we deduce that $f(u) = \mu_1[L, \Omega] u$,\footnote{Here we are following the notation in Appendix~\ref{app:uniqueness}, namely, $\mu_1[L, \Omega]$ denotes the principal eigenvalue of $L$ in $\Omega$ with the sign convention $L \varphi = -\mu \varphi$.} and hence $u$ is a principal eigenfunction of $L$.

Since $a_{ij} \in C(\overline{\Omega_k})$ and $\vv^{k}_i \in L^{\infty}(\Omega_k)$, by standard existence theory,
there is a principal eigenvalue $\mu_k := \mu_1[L_k, \Omega_k]$ and eigenfunction $\varphi_k \in W^{2,p}(\Omega_k)$ for all $p < \infty$,
satisfying $\varphi_k > 0$ in $\Omega_k$, $\|\varphi_k\|_{L^1(\Omega_k)} = \|u\|_{L^1(\Omega)}$, $-L_k \varphi_k = \mu_k \varphi_k$ in $\Omega_k$, and $\varphi_k = 0$ on $\partial \Omega_k$.
Moreover, recalling that $a_{ij} \in C^{0,1}(\overline{\Omega_k})$ and $\vv^k_i \in C^{\infty}(\overline{\Omega_k})$, we further have $\varphi_k \in W^{3,p}(\Omega_k)$ for $p < \infty$.
In particular, since $\varphi_k$ are stable, by Step 1 we have the bounds
\begin{equation}
\label{pug0}
\|\nabla \varphi_k\|_{L^{2+\gamma}(\Omega_k)} \leq C \|u\|_{L^1(\Omega)},
\end{equation}
and 
\begin{equation}
\label{pug2}
\|\varphi_k\|_{C^{\alpha}(\overline{\Omega_k})} \leq C \|u\|_{L^1(\Omega)} \quad \text{ if } n \leq 9.
\end{equation}

To deduce the final estimates it suffices to extract a subsequence converging to $u$.
For this, we essentially follow the proof of Theorem~2.1 in~\cite{BerestyckiNirenbergVaradhan}.
Namely, by Harnack inequality,
for $k \geq l+1$ we have
$\|\varphi_k\|_{L^{\infty}(\Omega_l)} \leq C_{l}\inf_{\Omega_{l}} \varphi_k \leq C_{l} \|\varphi_k\|_{L^1(\Omega_l)} \leq C_l \|u\|_{L^1(\Omega)}$,
hence, by interior estimates (up to a subsequence) 
$\varphi_k \to \varphi$ weakly in $W^{2, p}_{\rm loc}(\Omega)$
for some positive $\varphi \in W^{2,p}_{\rm loc}(\Omega)$.
Moreover, since $\varphi_k \in W^{1, 2+\gamma}_0(\Omega_k)$,
the extension $\varphi_k \chi_{\Omega_k}$ is bounded in $W_0^{1,2+\gamma}(\Omega)$
and by compactness $\varphi_k \chi_{\Omega_k} \to \varphi$ 
weakly in $W^{1, 2+\gamma}_0(\Omega)$, hence strongly in~$L^{2+\gamma}(\Omega)$.
In particular, by strong convergence $\|\varphi\|_{L^1(\Omega)} = \lim_{k} \|\varphi_k\|_{L^1(\Omega_k)} = \|u\|_{L^1(\Omega)}$
and by weak lower semicontinuity, from \eqref{pug0}, we deduce
\begin{equation}
\label{pug1}
\|\nabla \varphi \|_{L^{2+\gamma}(\Omega )} \leq C \|u\|_{L^1(\Omega)}.
\end{equation}
By \eqref{pug2}, using that $\varphi_k \to \varphi$ converges locally uniformly in $\Omega$, it is also clear that 
\begin{equation}
\label{pug3}
\|\varphi \|_{C^{\alpha}(\overline{\Omega})} \leq C \|u\|_{L^1(\Omega)} \quad \text{ if } n \leq 9.
\end{equation}

Passing to the limit in the equation, we see that
$\varphi \in W^{1,2+\gamma}_0(\Omega) \cap W^{2, p}_{\rm loc}(\Omega)$ solves $- L \varphi = \mu^{\star} \varphi$
where $\mu^{\star} = \lim_{k} \mu_1[L_k, \Omega_k]$.
In fact, $\mu^{\star} = \mu_1[L, \Omega]$ 
by the characterization of the principal eigenfunction and the maximum principle.

It follows that $u$ and $\varphi$ are both positive principal eigenfunctions of $L$ in $\Omega$
with the same norm
$\|u\|_{L^1(\Omega)} = \|\varphi\|_{L^1(\Omega)}$, hence $u = \varphi$.
Estimates~\eqref{pug1} and \eqref{pug3} already give the claim.
\end{proof}

\appendix

\medskip\medskip

\section{Approximating $C^{1,1}$ domains by smooth ones from the interior}
\label{app:approximation}

In this appendix, we show that bounded domains of class $C^{1,1}$ can be approximated from the interior by smooth sets satisfying uniform bounds.
This is a well-known result in the literature, and is valid more generally for domains of class $C^{k, \alpha}$ with $k \geq 1$ and $\alpha \in [0, 1]$.
We include a proof for the sake of completeness. 
Our proof follows the approach suggested by Gilbarg and Trudinger in~\cite[Problem~6.9]{GilbargTrudinger}.

First, we recall the definition of $C^{1,1}$ domains:

\begin{definition}
\label{def:domain}
A bounded domain $\Omega \subset \R^n$ is of class $C^{1,1}$ 
if at each point $x_0 \in \partial \Omega$ there is a ball $B = B_{\rho}(x_0)$
and a one-to-one mapping $\Psi$ of $B$ onto $U \subset \R^n$ such that
\begin{enumerate}[label=(\roman*)]
\item $\Psi (B \cap \Omega) \subset \R^n_{+}$;
\item $\Psi(B \cap \partial \Omega) \subset \partial \R^n_{+}$;
\item $\Psi \in C^{1,1}(\overline{B})$ and $\Psi^{-1} \in C^{1,1}(\overline{U})$.
\end{enumerate}
Equivalently, $\Omega$ is of class $C^{1,1}$ if each point of $\partial \Omega$ has a neighborhood in which $\partial \Omega$ is the graph of a $C^{1,1}$ function of $n-1$ of the coordinates.
\end{definition}

Every such domain can be written as the positive set of a $C^{1,1}$ function.
\begin{lemma}
\label{lemma:domain}
Let $\Omega \subset \R^n$ be a bounded domain of class $C^{1,1}$.
Then there is a 
function $\Phi \in C^{1,1}(\R^n)$
such that
$\Omega = \{\Phi > 0\}$,
$\Phi = 0$ on $\partial \Omega$, 
and $\nabla \Phi(x) \neq 0$ for all $x \in \partial \Omega$.
\end{lemma}
\begin{proof}
By compactness, 
$\partial \Omega$ may be covered by finitely many balls $\{B_j = B_{\rho_j}(x_{j})\}_{j = 1}^N$, with $x_j \in \partial \Omega$ and $\rho_j > 0$,
such that there are flattening maps $\Psi_{j} \in C^{1,1}(\overline{B_j})$ 
as in Definition~\ref{def:domain}.

Let $\rho > 0$ be sufficiently small so that the set
$B_0 := \{ x \in \Omega \colon{\rm dist } (x, \partial \Omega) > \rho  \}$
satisfies $\overline{\Omega} \subset \cup_{j = 0}^{N} B_j$,
and consider a partition of unity $\{\eta_j\}_{j = 0}^{N}$ subordinated to the covering $\{B_j\}_{j = 0}^{N}$.
The function
\[
\Phi = \eta_0 + \sum_{j = 1}^{N} \eta_j \Psi_j^{n}
\]
now satisfies the desired properties.
\end{proof}
\begin{remark}
Notice that, by construction, the function $\Phi$ above
is compactly supported and
takes negative values in a bounded neighborhood of $\partial \Omega$ outside $\overline{\Omega}$.
\end{remark}

Regularizing $\Phi$ and taking appropriate superlevel sets, 
we obtain the approximation.

\begin{lemma}
Let $\Omega = \{\Phi > 0\}\subset \R^n$ be a bounded domain of class $C^{1,1}$, with $\Phi \in C^{1,1}(\R^n)$ as in Lemma~\ref{lemma:domain} above.
Then, there is an exhaustion
of $\Omega$ by smooth sets $\Omega_k = \{\Phi_k > 0\}$,\footnote{By an exhaustion we mean that $\overline{\Omega_k} \subset \Omega_{k+1} \subset \Omega$ and $\Omega = \cup_{k} \Omega_k$.}
where the functions $\Phi_k \in C^{\infty}(\R^n)$ satisfy
\[
\|\nabla \Phi_k\|_{C^{1}(\R^n)} 
+ \||\nabla \Phi_k|^{-1}\|_{L^{\infty}(\partial \Omega_k)} \leq C,
\]
for some constant $C$ depending only on $\Phi$ and $\Omega$. 
Moreover, we have that 
$\partial \Omega_k \to \partial \Omega$ in the sense of the Hausdorff distance.\footnote{
By this we mean that $\max \{\sup_{x \in \partial \Omega} {\rm dist} (x, \partial \Omega_k), \sup_{x \in \partial \Omega_k} {\rm dist} (x, \partial \Omega)\} \to 0$ as $k \to \infty$.}
\end{lemma}
\begin{proof}

Since $\Omega = \{\Phi > 0\}$
and $\nabla \Phi \neq 0$ on $\partial \Omega$,
by continuity of the gradient, $\Phi$ is comparable to 
the distance function
$d(x) := {\rm dist }(x, \partial \Omega)$ in $\overline{\Omega}$, i.e.,
\begin{equation}
\label{ineq:comp}
L^{-1} d(x) \leq \Phi(x) \leq L d(x) \quad \text{ for } x \in \overline{\Omega},
\end{equation}
for some 
$L \geq 1$.
For $x$ outside $\Omega$, we define
$d(x) := - {\rm dist}(x, \partial \Omega) < 0$.

Consider a mollifying sequence $(\eta^{\varepsilon})_{\varepsilon > 0}$ with $\supp \eta^{\varepsilon} \subset B_{\varepsilon}$.
We will take the functions
\begin{equation}
\label{def:phiseq}
\Phi_k := \Phi * \eta^{\varepsilon_k} - 2L \varepsilon_k
\end{equation}
for an appropriate sequence $\varepsilon_k \downarrow 0$.
Since $\|\nabla \Phi_k\|_{C^1} = \|\nabla \Phi_k\|_{L^{\infty}} + \|D^2 \Phi_k\|_{L^{\infty}}$,
recalling that $\|D^2\Phi_k\|_{L^{\infty}(\R^n)} = [\nabla \Phi_k]_{C^{0,1}(\R^n)}$ (because $\Phi_k$ is $C^{\infty}$) and $\Phi \in C^{1,1}$,
the uniform bounds for $\|\nabla \Phi_k\|_{C^1}$ stated in the lemma hold by the standard properties of convolutions.

Given $\delta > 0$, from \eqref{ineq:comp} we see that
\[
\Phi(x) > \delta/ L \quad 
\text{ for } x \in 
 \{x \in \R^n \colon d(x) > \delta \} = \{d > \delta\},
\]
and taking the convolution with $\eta^{\varepsilon}$, it follows that
\begin{equation}
\label{ineq:inside}
\Phi*\eta^{\varepsilon}(x) > \delta/ L \quad 
\text{ for } x \in  \{d > \delta + \varepsilon\}.
\end{equation}
Similarly, for $\widetilde{\delta} > 0$ we have that
\[
\Phi(x) \leq L \widetilde{\delta} \quad 
\text{ for } x \in \{d \leq \widetilde{\delta} \},
\]
and regularizing we obtain
\begin{equation}
\label{ineq:outside}
\Phi*\eta^{\varepsilon}(x)
\leq L \widetilde{\delta} \quad 
\text{ for } 
x \in  \{ d \leq \widetilde{\delta} - \varepsilon \}.
\end{equation}
Letting $\widetilde{\delta} = \delta/L^2$, 
since $L \widetilde{\delta} = \delta/L$,
by \eqref{ineq:inside} and \eqref{ineq:outside} we have that
\[
\{d > \delta + \varepsilon\} \subset \{\Phi*\eta^{\varepsilon} > \delta/L\} \subset \{d > \delta/L^2 - \varepsilon\},
\]
and the choice $\delta = 2 L^2 \varepsilon$ now yields the inclusions
\begin{equation}
\label{inclusions}
\{d > \varepsilon(2L^2 + 1)\} \subset \{\Phi*\eta^{\varepsilon} > 2 L \varepsilon\} \subset \{d > \varepsilon\}.
\end{equation}

Next, we construct the sequence $\varepsilon_k$ in \eqref{def:phiseq} inductively.
Fix $\varepsilon_1 > 0$ small.
For $k \geq 1$, we define $\varepsilon_{k+1} := 
\varepsilon_{k}/ \{2 (2 L^2 + 1)\} = \varepsilon_1/\{2 (2 L^2 + 1)\}^k$.
Hence, by \eqref{inclusions}, the sets
$\Omega_k := \{\Phi_k > 0\} = \{\Phi *\eta^{\varepsilon_k} > 2 L \varepsilon_k\}$
satisfy 
\begin{equation}
\label{distbdy}
\overline{\Omega_{k}} \subset \{d \geq \varepsilon_k\} \subset \{d > \varepsilon_{k+1} (2L^2 + 1)\} \subset \Omega_{k+1} \subset \Omega.
\end{equation}
They clearly exhaust $\Omega$, since $\varepsilon_k \downarrow 0$ and thus
$\Omega = \cup_{k}\{d \geq \varepsilon_{k-1}\} \subset \cup_{k} \Omega_k$.
Furthermore, the inclusions~\eqref{distbdy} show that $\partial \Omega_{k}$ is at a Hausdorff distance of at most $\varepsilon_k(2L^2+1)$ from $\partial \Omega$,
and hence $\partial \Omega_k \to \partial \Omega$ with respect to this distance.

It remains to prove the lower bound for $|\nabla \Phi_k|$ on $\partial \Omega_k$, which will also show the boundary $\partial \Omega_k$ to be smooth.
Let $x \in \R^n$ with $|d(x)| = {\rm dist }(x, \partial \Omega) = |x-x_0|$ for some $x_0 \in \partial \Omega$.
Since
\[
\textstyle
\nabla \Phi(x) \cdot \frac{\nabla \Phi(x_0)}{|\nabla \Phi(x_0)|} 
\geq |\nabla \Phi(x_0)| - [\nabla \Phi]_{C^{0,1}(\R^n)} |d(x)|, 
\]
we have the lower bound
\[
\textstyle
\nabla \Phi(x) \cdot \frac{\nabla \Phi(x_0)}{|\nabla \Phi(x_0)|}  \geq 
\frac{1}{2}\||\nabla \Phi|^{-1}\|_{L^{\infty}(\partial \Omega)}^{-1}
\quad \text{ for } x \in \{- \rho < d < \rho\},
\]
where $\rho > 0$ depends only on $\||\nabla \Phi|^{-1}\|_{L^{\infty}(\partial \Omega)}$ and $[\nabla \Phi]_{C^{0,1}(\R^n)}$.
Taking the convolution with $\eta^{\varepsilon_k}$ in this last inequality, we obtain
\begin{equation}
\label{tis}
\textstyle
\nabla \Phi_k(x) \cdot \frac{\nabla \Phi(x_0)}{|\nabla \Phi(x_0)|}  \geq 
\frac{1}{2}\||\nabla \Phi|^{-1}\|_{L^{\infty}(\partial \Omega)}^{-1}
\quad \text{ for } x \in \{- (\rho - \varepsilon_k) < d < \rho - \varepsilon_k \}.
\end{equation}
By \eqref{distbdy}, the boundary 
$\partial \Omega_{k}$
is at a distance of at most $\varepsilon_{k} (2L^2 + 1)$ from $\partial \Omega$.
Hence, choosing $\varepsilon_1 > 0$ sufficiently small such that $\rho - \varepsilon_1 > \varepsilon_1(2L^2 + 1)$,
from \eqref{tis} and the definition of $\varepsilon_k = \varepsilon_1/\{2(2L^2+1)\}^{k-1}$
we deduce
\[
\textstyle
|\nabla \Phi_k(x)| \geq \nabla \Phi_k(x) \cdot \frac{\nabla \Phi(x_0)}{|\nabla \Phi(x_0)|} \geq 
\frac{1}{2}\||\nabla \Phi|^{-1}\|_{L^{\infty}(\partial \Omega)}^{-1}
\quad \text{ on } \partial \Omega_k,
\]
and therefore $\||\nabla \Phi_k|^{-1}\|_{L^{\infty}(\partial \Omega_k)} \leq 2 \||\nabla \Phi|^{-1}\|_{L^{\infty}(\partial \Omega)}$, which concludes the proof.
\end{proof}

We conclude this appendix with a description of the flattening procedure
and a technical lemma used in the proof of Theorem~\ref{thm:c11} in Section~\ref{section:approximation}.
\medskip

Let $x_0 \in \partial \Omega$.
Rotating the coordinate axes, we may assume that $\nabla \Phi(x_0) = \partial_n \Phi(x_0) e_n$, with $|\nabla \Phi(x_0)| = \partial_n \Phi(x_0) > 0$.
Writing $x = (x', x_n) \in \R^{n-1} \times \R$, the map $\Psi \colon \R^n\to\R^n$ given by
$\Psi(x) := \big((x-x_0)', \frac{\Phi(x)}{\partial_n \Phi(x_0)} \big)$
is a local diffeomorphism around $x_0$ which flattens out the boundary $\partial \Omega$.
More precisely, by a quantitative version of the Inverse Function Theorem, we have the following:
\begin{lemma}
\label{lemma:prelim}
There are (small) numbers 
$0 < R_1 < R_2$ and $\rho > 0$
depending only on 
$\|\nabla \Phi\|_{C^{0,1}(\R^n)}$
and $\||\nabla \Phi|^{-1}\|_{L^{\infty}(\partial \Omega)}$ such that
\begin{equation}
\label{eq:claim}
\Psi(B_{R_1}(x_0) \cap \Omega) \subset B_{\rho/2}^{+} \subset B_{\rho}^{+} \subset \Psi(B_{R_2}(x_0) \cap \Omega).
\end{equation}
\end{lemma}

\begin{proof}
By translation, we may assume that $x_0 = 0 \in \partial \Omega$.
Since the map $\Psi$ 
satisfies $\Psi(0) = 0$ and $D \Psi(0) = \id$,
choosing $R_2 > 0$ small such that
\begin{equation}
\label{cond1}
|D \Psi(x) - D \Psi(z)| \leq 1/2 \quad \text{ for all } x, z \in B_{R_2},
\end{equation}
by Lemma~1.3 in~\cite[Chapter XIV]{LangAnalysis} we deduce that for all $y \in B_{R_2/2}$
there is a unique $x \in B_{R_2}$ such that $\Psi(x) = y$.
Thus, we obtain the second inclusion
\begin{equation}\label{eq:2nd}
B^{+}_{R_2/2} \subset \Psi(B_{R_2} \cap \Omega).
\end{equation}
Using that $[D \Psi]_{C^{0,1}(\R^n)} \leq [\nabla \Phi]_{C^{0,1}(\R^n)} \||\nabla \Phi|^{-1}\|_{L^{\infty}(\partial \Omega)}$,
it is easy to check that condition~\eqref{cond1} is fulfilled if
\begin{equation}
\label{thisone}
R_2 \leq (4 [\nabla \Phi]_{C^{0,1}(\R^n)} \||\nabla \Phi|^{-1}\|_{L^{\infty}(\partial \Omega)})^{-1}.
\end{equation}

To show the first inclusion in \eqref{eq:claim}, 
we proceed as above but considering the inverse map $\Psi^{-1}$ instead of $\Psi$.
If $R_1 > 0$ is such that
\begin{equation}
\label{cond15}
|D \Psi^{-1}(\widetilde{x}) - D \Psi^{-1}(\widetilde{z})| \leq 1/2 \quad \text{ for all } \widetilde{x}, \widetilde{z} \in B_{2R_1},
\end{equation}
then, again
by Lemma~1.3 in~\cite[Chapter XIV]{LangAnalysis} we deduce
\begin{equation}
\label{claim2}
\Psi(B_{R_1} \cap \Omega) \subset B_{2 R_1}^{+}.
\end{equation}
For \eqref{cond15} to hold, it suffices to take $R_1 > 0$ sufficiently small such that
\begin{equation}
\label{cond3}
[D \Psi^{-1}]_{C^{0,1}(B_{2 R_1})} 4 R_1 \leq 1/2.
\end{equation}
Hence, to quantify the smallness of $R_1$, we have to estimate $[D \Psi^{-1}]_{C^{0,1}(B_{2 R_1})}$.

Let $\widetilde{x} = \Psi(x)$ and  $\widetilde{y}= \Psi(y) $ in $B_{2 R_1}$.
If $R_1 \leq R_2$, with $R_2$ satisfying \eqref{thisone},
then
$\frac{1}{\partial_n \Phi(x)} 
\leq 2\||\nabla \Phi|^{-1}\|_{L^{\infty}(\partial \Omega)} $
and hence
\[
\begin{split}
&|D \Psi^{-1}(\widetilde{x})  - D \Psi^{-1}(\widetilde{y}) |^2 \\
&\leq 
2\frac{|\nabla' \Phi(x) - \nabla' \Phi(y)|^2}{|\partial_n \Phi(x)|^2}
+\frac{\left(2|\nabla' \Phi(y)|^2 + |\partial_n \Phi(0)|^2 \right)|\partial_n \Phi(x) - \partial_n \Phi(y)|^2}{|\partial_n \Phi(x)|^2 |\partial_n \Phi(y)|^2}\\
&\leq 8 ( 1
+ 6 \|\nabla \Phi\|^2_{L^{\infty}(\R^n)} \||\nabla \Phi|^{-1}\|^2_{L^{\infty}(\partial \Omega)} )
 \||\nabla \Phi|^{-1}\|^2_{L^{\infty}(\partial \Omega)}
 [\nabla \Phi]^2_{C^{0,1}(\R^n)}
|x - y|^2.
\end{split}
\]
Moreover,
since
\[
|x - y| \leq [D \Psi^{-1}]_{L^{\infty}(B_{2R_1})} |\widetilde{x} - \widetilde{y}| \leq (1 + 8 \|\nabla \Phi\|_{L^{\infty}}^2 \||\nabla \Phi|^{-1}\|_{L^{\infty}(\partial \Omega)}^2)^{1/2} |\widetilde{x} - \widetilde{y}|,
\]
by the above we conclude
\[
[D \Psi^{-1}]_{C^{0,1}(B_{2 R_1})} \leq 2 \sqrt{2} (1 + 8 \|\nabla \Phi\|_{L^{\infty}}^2 \||\nabla \Phi|^{-1}\|_{L^{\infty}(\partial \Omega)}^2)
 \||\nabla \Phi|^{-1}\|_{L^{\infty}(\partial \Omega)} 
 [\nabla \Phi]_{C^{0,1}(\R^n)}
\]
for all $R_1 \leq R_2$ satisfying \eqref{thisone}.
Therefore, if
\begin{equation}
\label{thatone}
R_1 \leq \left( 16 \sqrt{2} 
\big( 1
+ 8 \|\nabla \Phi\|^2_{L^{\infty}(\R^n)} \||\nabla \Phi|^{-1}\|^2_{L^{\infty}(\partial \Omega)} \big) 
 \||\nabla \Phi|^{-1}\|_{L^{\infty}(\partial \Omega)}
 [\nabla \Phi]_{C^{0,1}(\R^n)}
\right)^{-1},
\end{equation}
then \eqref{cond3} holds and \eqref{claim2} follows. 

Finally, choosing $R_1$, $R_2 > 0$ satisfying~\eqref{thisone}, \eqref{thatone}, and $8 R_1 \leq R_2$,
we may take $\rho = 4 R_1$
and the inclusions \eqref{eq:2nd} and \eqref{claim2} then yield the claim~\eqref{eq:claim}.
\end{proof}

\section{On the uniqueness of stable solutions}
\label{app:uniqueness}

Here, we prove the uniqueness of stable solutions to
nonvariational equations involving convex nonlinearities.
For this, we employ some fundamental results of Berestycki, Nirenberg, and Varadhan~\cite{BerestyckiNirenbergVaradhan}
on the principal eigenfunction.
Compare the following statement with Proposition 1.3.1 in Dupaigne's book~\cite{Dupaigne}:

\begin{proposition}
\label{prop:uniqueness}
Given $\Omega\subset \R^n$ a bounded domain, let $u_1$, $u_2 \in C^{0}(\overline{\Omega}) \cap W^{2,n}_{\rm loc}(\Omega)$ 
be two stable solutions of the equation $- L u = f(u)$ in $\Omega$, with $u = 0$ on $\partial \Omega$.

Assume that $f \in C^1(\R)$ is convex.

Then either $u_1 = u_2$ or $f(u) = \mu_1[L, \Omega] u$ on the ranges of $u_1$ and $u_2$.
\end{proposition}
\begin{remark}
Here $\mu_1[L, \Omega]$ denotes the principal (or smallest) eigenvalue of $L$ in $\Omega$,
with the sign convention $-L \varphi = \mu_1 \varphi$.
It is characterized by (see \cite{BerestyckiNirenbergVaradhan})
\[
\mu_1[L, \Omega] =\sup\left\{
\mu \colon \text{ there is a function } \varphi > 0 \in W^{2,n}_{\rm loc}(\Omega) \text{ satisfying } L \varphi + \mu \varphi \leq 0 \text{ in } \Omega
\right\}.
\]
Moreover, $L$  can be any uniformly elliptic second order operator.
In particular, we allow it to have zero order terms.
\end{remark}
\begin{remark}
In the proof of Theorem~\ref{thm:c11} above,
we only need a weaker version of Proposition~\ref{prop:uniqueness}.
Namely, we could assume additionally that $u_1 \leq u_2$,
which admits a shorter proof.
However, the present statement might be more useful in applications.
\end{remark}

\begin{proof}
Assume that $u_1 \neq u_2$ and consider the difference $w := u_2 - u_1$.
Let 
\[
\Omega^{+} := \{x \in \Omega \colon u_2(x) > u_1(x)\} = \{w > 0\},
\]
and assume that
$\Omega^{+} \neq 0$ (otherwise we exchange the roles of $u_1$ and $u_2$).
By convexity
\[
- L w = f(u_2) - f(u_1)\leq f'(u_2) w,
\]
and hence
\begin{equation}
\label{she}
J_{u_2} w = L w + f'(u_2) w \geq 0 \quad \text{ in } \Omega.
\end{equation}

By the monotonicity of the principal eigenvalue with respect to the domain, 
since $u_2$ is stable,
we have that
\begin{equation}
\label{she1}
\mu_1[J_{u_2}, \Omega^{+}] \geq \mu_1[J_{u_2}, \Omega] \geq 0.
\end{equation}
Since $w > 0$ in $\Omega^{+}$, by \eqref{she} and \eqref{she1}
it follows that 
\[
\left\{
\begin{array}{ll}
J_{u_2} w +\mu_1[J_{u_2}, \Omega^{+}] w
\geq 0 & \text{ in } \Omega^{+}\\
w = 0 & \text{ on } \partial \Omega^{+}.
\end{array}
\right.
\]
Applying 
\cite[Corollary~2.2]{BerestyckiNirenbergVaradhan},
we deduce that $w$ is a positive principal eigenfunction of $J_{u_2}$ in $\Omega^{+}$, 
that is, satisfying
\begin{equation}
\label{she2}
J_{u_2} w 
+\mu_1[J_{u_2}, \Omega^{+}] w
= 0
\quad \text{ in } \Omega^{+}.
\end{equation}
Using the equation $- L w = f(u_2) - f(u_1)$ in $\Omega^{+}$, from \eqref{she2} we see that
\begin{equation}
\label{anticonvex}
f(u_2) - f(u_1) - f'(u_2) (u_2 - u_1)=  
\mu_1[J_{u_2}, \Omega^{+}]
(u_2 - u_1) \geq 0 \quad \text{ in } \Omega^{+}.
\end{equation}
By convexity, we also have the opposite inequality $f(u_2) - f(u_1) - f'(u_2) (u_2 - u_1) \leq 0$.
Hence, by~\eqref{anticonvex}, we conclude $\mu_1[J_{u_2}, \Omega^{+}] = 0$ and the nonlinearity $f$ is affine in the union of intervals $[u_1(x), u_2(x)]$ with $x \in \Omega^{+}$.

If $f$ is of the form $f(u) = a u + b$ in the ranges above,
then $- L w = a w$ in $\Omega^{+}$, with $w > 0$, and therefore $a = \mu_1[L, \Omega^{+}]$.
Now, to have nontrivial solutions of $- L u = \mu_1[L, \Omega^{+}] u + b$,
the Fredholm alternative forces $b = 0$.
We conclude that $f(u) = \mu_1[L, \Omega^{+}] u$ in the ranges of $u_1$ and $u_2$ in $\Omega^{+}$.

If $\Omega^{-} := \{w < 0\} \neq \emptyset$,
then, arguing as above with $- w$ in place of $w$, we deduce that
\[
J_{u_1}w + \mu_1[J_{u_1}, \Omega^{-}] w = 0 \quad \text{ in } \Omega^{-},
\]
with $\mu_1[J_{u_1}, \Omega^{-}] = 0$
and $f(u) = \mu_1[L, \Omega^{-}] u $ in the ranges of $u_2$ and $u_1$ in $\Omega^{-}$.

The regularity of $f$ and the continuity of the solutions now forces $\mu_1[L, \Omega^{+}] = \mu_1[L, \Omega^{-}]$.
Now, on the one hand,
by the stability of $u_2$ we have that
\[
0 \leq \mu_1[J_{u_2}, \Omega] 
= 
\mu_1\big[L + \mu_1[L, \Omega^{+}], \Omega \big] =\mu_1[L, \Omega] - \mu_1[L, \Omega^{+}].
\]
On the other hand, 
since $\Omega^{+} \subsetneq \Omega$ by assumption,
the strict monotonicity of $\mu_1$ yields
\[
\mu_1[L, \Omega] - \mu_1[L, \Omega^{+}] < 0.\footnote{To show that $\mu_1[L, \Omega] < \mu_1[L, \Omega^{+}]$, first notice that we already have $\mu_1[L, \Omega]\leq \mu_1[L, \Omega^{+}]$ by definition of $\mu_1$ in $\Omega^+$. Suppose that $\mu_1[L, \Omega]  = \mu_1[L, \Omega^{+}]$ and consider $\varphi_1$ a positive principal eigenfunction of $L$ in $\Omega$. In particular $L \varphi_1 + \mu_1[L, \Omega^{+}]\varphi_1 = 0$ in $\Omega^{+}$ and $\varphi_1 > 0$ in $\Omega^{+}$. By  Corollary~2.1 in \cite{BerestyckiNirenbergVaradhan} we conclude $\varphi_1 = 0$ on $\partial \Omega^{+} \cap \Omega \neq \emptyset$, contradicting the positivity in $\Omega$.}
\]
This contradiction forces either $\Omega^{+}$ or $\Omega^{-}$ to be empty,
and therefore $f(u) = \mu_1[L, \Omega] u$ in the ranges of $u_1$ and $u_2$, which was the claim.
\end{proof}

\section*{Acknowledgments}

The author wishes to thank Xavier Cabr\'{e} for useful discussions on the topic of this article, as well as for his encouragement over the years.
He also thanks the anonymous referee for their suggestions, which improved the presentation of the paper.


\begin{thebibliography}{99}








\bibitem{BerestyckiKiselevNovikovRyzhik}
\textit{H.~Berestycki}, \textit{A.~Kiselev}, \textit{A.~Novikov}, and \textit{L.~Ryzhik},
The explosion problem in a flow, J. Anal. Math. \textbf{110} (2010), 31--65.

\bibitem{BerestyckiNirenbergVaradhan}
\textit{H.~Berestycki}, \textit{L.~Nirenberg}, and \textit{S.~R.~S.~Varadhan},
The principal eigenvalue and maximum principle for second-order elliptic operators in general domains, Comm. Pure Appl. Math. \textbf{47} (1994), 47--92.

\bibitem{Brezis-IFT}
\textit{H.~Brezis},
Is there failure of the inverse function theorem?,
in: Morse theory, minimax theory and their applications to nonlinear differential equations \textbf{1} (Beijing 1999), Int. Press, Somerville, MA (2003), 22--23.

\bibitem{BrezisVazquez}
\textit{H.~Brezis} and \textit{J.~L.~V\'{a}zquez},
Blow-up solutions of some nonlinear elliptic problems, Rev. Mat. Univ. Complut. Madrid \textbf{47} (1997), 443--469.

\bibitem{CabreRadial}
\textit{X.~Cabr\'{e}}, 
Estimates controlling a function by only its radial derivative and applications to stable solutions of elliptic equations, preprint 2022, \url{https://arxiv.org/abs/2211.13033}.

\bibitem{Cabre-Dim4}
\textit{X.~Cabr\'{e}}, 
Regularity of minimizers of semilinear elliptic problems up to dimension 4, Comm. Pure Appl. Math. \textbf{63} (2010), 1362--1380.

\bibitem{CabreCapella-Radial}
\textit{X.~Cabr\'{e}} and \textit{A.~Capella}, 
Regularity of radial minimizers and extremal solutions of semilinear elliptic equations, J. Funct. Anal. \textbf{238} (2006), 709--733.

\bibitem{CabreFigalliRosSerra}
\textit{X.~Cabr\'{e}}, \textit{A.~Figalli}, \textit{X.~Ros-Oton}, and \textit{J.~Serra},
Stable solutions to semilinear elliptic equations are smooth up to dimension 9, Acta Math. \textbf{224} (2020), 187--252.

\bibitem{CabreMiraglioSanchon}
\textit{X.~Cabr\'{e}}, \textit{P.~Miraglio}, and \textit{M.~Sanch\'{o}n},
Optimal regularity of stable solutions to nonlinear equations involving the {$p$}-{L}aplacian, Adv. Calc. Var. \textbf{15} (2022), 749--785.

\bibitem{CabreRosOton-DoubleRev}
\textit{X.~Cabr\'{e}} and \textit{X.~Ros-Oton},
Regularity of stable solutions up to dimension {$7$} in domains of double revolution, Comm. Partial Differential Equations \textbf{38} (2013), 135--154.

\bibitem{CostadeSouzaMontenegro}
\textit{F.~Costa}, \textit{G.~F.~de~Souza}, and \textit{M.~Montenegro},
Extremal solutions of strongly coupled nonlinear elliptic systems and {$L^\infty$}-boundedness, J. Math. Anal. Appl. \textbf{513} (2022), Paper No. 126225, 28.

\bibitem{CowanGhoussoub}
\textit{C.~Cowan} and \textit{N.~Ghoussoub},
Regularity of the extremal solution in a {MEMS} model with advection, Methods Appl. Anal. \textbf{15} (2008), 355--360.

\bibitem{CrandallRabinowitz}
\textit{M.~G.~Crandall} and \textit{P.~H.~Rabinowitz},
Some continuation and variational methods for positive solutions of nonlinear elliptic eigenvalue problems, Arch. Rational Mech. Anal. \textbf{58} (1975), 207--218.

\bibitem{Dupaigne}
\textit{L.~Dupaigne},
Stable solutions of elliptic partial differential equations, Chapman and Hall/CRC, 2011.

\bibitem{ErnetaBdy1}
\textit{I.~U.~Erneta}, 
Energy estimate up to the boundary for stable solutions to semilinear elliptic problems, J. Differential Equations \textbf{378} (2024), 204--233.

\bibitem{ErnetaInterior}
\textit{I.~U.~Erneta},
Stable solutions to semilinear elliptic equations for operators with variable coefficients, Commun. Pure Appl. Anal. \textbf{22} (2023), 530--571.

\bibitem{Gelfand}
\textit{I.~M.~Gel'fand},
Some problems in the theory of quasilinear equations, Amer. Math. Soc. Transl. (2) \textbf{29} (1963), 295--381.

\bibitem{GilbargTrudinger}
\textit{D.~Gilbarg} and \textit{N.~S.~Trudinger},
Elliptic partial differential equations of second order, 2nd ed., Springer Berlin, New York 2001.

\bibitem{JosephLundgren}
\textit{D.~D.~Joseph} and \textit{T.~S.~Lundgren},
Quasilinear {D}irichlet problems driven by positive sources, Arch. Rational Mech. Anal. \textbf{49} (1973), 241--269.

\bibitem{LangAnalysis}
\textit{S.~Lang},
Real and functional analysis, 3nd ed., Graduate Texts in Mathematics Springer-Verlag, New York 1993.

\bibitem{Martel}
\textit{Y.~Martel},
Uniqueness of weak extremal solutions of nonlinear elliptic problems, Houston J. Math. \textbf{23} (1997), 161--168.

\bibitem{Nedev}
\textit{G.~Nedev},
Regularity of the extremal solution of semilinear elliptic equations, C. R. Acad. Sci. Paris S{\'e}r. I Math. \textbf{330} (2000), 997--1002.

\bibitem{SternbergZumbrun1}
\textit{P.~Sternberg} and \textit{K.~Zumbrun},
Connectivity of phase boundaries in strictly convex domains, Arch. Rational Mech. Anal. \textbf{141} (1998), 375--400.

\bibitem{Villegas}
\textit{S.~Villegas}, 
Boundedness of extremal solutions in dimension 4, Adv. Math. \textbf{235} (2013), 126--133.

\end{thebibliography}
\end{document}